\documentclass[10pt,twoside, a4paper, english, reqno]{amsart}
\usepackage[dvips]{epsfig}
\usepackage{amscd}
\usepackage{a4wide}
\usepackage{amssymb}
\usepackage{amsthm}
\usepackage{amsmath}
\usepackage{latexsym}

\usepackage{upref}
\usepackage[colorlinks,backref]{hyperref}
\usepackage{color}
\usepackage[usenames,dvipsnames]{xcolor}

\usepackage{pdfsync}

\theoremstyle{plain}
\newtheorem{thm}{Theorem}[section]
\theoremstyle{plain}
\newtheorem{lem}[thm]{Lemma}
\newtheorem{prop}[thm]{Proposition}

\theoremstyle{definition}
\newtheorem{defi}{Definition}[section]
\newtheorem{rem}{Remark}[section]

\newtheorem*{maintheorem*}{Main Theorem}
\newtheorem*{maincorollary*}{Main Corollary}

\newenvironment{Assumptions}
{
\setcounter{enumi}{0}

\begin{enumerate}}
{\end{enumerate} }

\newenvironment{Assumptions2}
{
\setcounter{enumi}{0}

\begin{enumerate}}
{\end{enumerate} }

\def\cprime{$'$}
\newcommand{\norm}[1]{\left\|#1\right\|}

\newcommand{\con} {\ast}
\newcommand{\R}{\ensuremath{\mathbb{R}}}

\newcommand{\rd}{\ensuremath{\mathbb{R}^d}}
\newcommand{\supp}{\ensuremath{\mathrm{supp}\,}}
\newcommand{\goto}{\ensuremath{\rightarrow}}
\newcommand{\grad}{\ensuremath{\nabla}}
\newcommand{\eps}{\ensuremath{\varepsilon}}

\newcommand{\E}{\mathbb{E}}

\numberwithin{equation}{section} \allowdisplaybreaks

\title[Continuous dependence]
{Continuous dependence estimate for a degenerate parabolic-hyperbolic equation with L\'{e}vy noise}

\date{\today}

\keywords{Degenerate Parabolic-Hyperbolic Equation; L\'{e}vy Noise; Young measures; Rate of Convergence; Fractional BV Estimates.}

\thanks{Supported in part by Indo-French Centre for Applied Mathematics (IFCAM) and Institute for the Sustainable Engineering of Fossil Resources (ISIFoR)}

\author[Ujjwal Koley]{Ujjwal Koley}
\address[Ujjwal Koley]{\newline
	Centre for Applicable Mathematics,
	Tata Institute of Fundamental Research,
	P.O.\ Box 6503, GKVK Post Office,
	Bangalore 560065, India}
\email[]{ujjwal@math.tifrbng.res.in}

\author[Ananta K. Majee]{Ananta K. Majee}
\address[Ananta K. Majee]{\newline
	Mathematisches Institut,
	Universit\"{a}t T\"{u}bingen,
	Auf der Morgenstelle 10, D-72076 T\"{u}bingen, Germany}
\email[]{majee@na.uni-tuebingen.de}

\author[Guy Vallet]{Guy Vallet}
\address[Guy Vallet]{\newline 
	LMAP UMR- CNRS 5142, IPRA BP 1155, 64013 Pau Cedex, France}
\email[]{guy.vallet@univ-pau.fr}

\begin{document}
\begin{abstract}
In this article, we are concerned with a multidimensional degenerate parabolic-hyperbolic equation driven by L\'{e}vy processes. Using bounded variation (BV) estimates for vanishing viscosity approximations, we derive an explicit continuous dependence estimate on the nonlinearities of the entropy solutions under the assumption that L\'{e}vy noise depends only on the solution. This result is used to show the error estimate for the stochastic vanishing viscosity method. In addition, we establish fractional BV estimate for vanishing viscosity approximations in case the noise coefficients depend on both the solution and spatial variable.
\end{abstract}
\maketitle
\tableofcontents
\section{Introduction}
The last couple of decades have witnessed remarkable advances on the larger area of stochastic partial differential equations that are driven by L\'{e}vy noise. An worthy reference on this subject is \cite{peszat}. However, very little is available on the specific problem of degenerate parabolic-hyperbolic equation with L\'{e}vy noise, and there are still a number of issues waiting to be explored. In this paper, we aim at deriving continuous dependence estimates based on nonlinearities for stochastic degenerate parabolic-hyperbolic equation driven by multiplicative L\'{e}vy noise. A formal description of our problem requires a filtered probability space $\big(\Omega, \mathbb{P}, \mathcal{F}, \{\mathcal{F}_t\}_{t\ge 0} \big)$, and we are interested in an $L^2$-valued predictable process $u(t,\cdot)$
which satisfies the following Cauchy problem
\begin{equation}
\label{eq:stoc_con_brown}
\begin{cases} 
du(t,x)- \mbox{div} f(u(t,x)) \,dt-\Delta A(u(t,x))\,dt =
\sigma(u(t,x))\,dW(t) + \int_{|z|>0} \eta(u(t,x);z)\widetilde{N}(dz,dt), & x \in \Pi_T, \\
u(0,x) = u_0(x), &  x\in \R^d,
\end{cases}
\end{equation}
where $\Pi_T= \R^d \times (0,T)$ with $T>0$ fixed, $u_0:\R^d \mapsto \R$ is the given initial
function, $f:\R \mapsto \R^d$ is a given (sufficiently smooth) scalar valued flux function 
(see Section~\ref{sec:tech} for the complete list of assumptions), and $A:\R \mapsto \R$ is a given nonlinear diffusion. Regarding this, the basic assumption is that $A(\cdot)$ is nondecreasing with $A(0)=0$. Moreover, \eqref{eq:stoc_con_brown} is allowed to be strongly degenerate in the sense that $A^{\prime}(\cdot)$ is allowed to be zero on an interval, see \cite{carrillo_1999}. 
Furthermore, $W(t)$ is a real valued Brownian noise and $ \widetilde{N}(dz,dt)= N(dz,dt)-m(dz)\,dt $, 
where $N$ is a Poisson random measure on $\R\times (0,\infty)$ with intensity measure $m(dz)$, a Radon measure on $\R \setminus \lbrace0\rbrace$
with a possible singularity at $z=0$ satisfying $\int_{|z|>0} (1\wedge |z|^2)\,m(dz) <  + \infty$.\footnote{Here we denote $ x \wedge y :=\min{\lbrace x,y \rbrace}$}
Finally, $ u \mapsto \sigma (u)$ and $(u,z)\mapsto \eta(u,z)$ are given real valued functions signifying
the multiplicative nature of the noise.

The equation \eqref{eq:stoc_con_brown} could be viewed as a stochastic perturbation of parabolic-hyperbolic equation. Equations of this type model the phenomenon of convection-diffusion of ideal 
fluids and therefore arise in a wide variety of important applications, including for instance two or three phase flows in porous media \cite{EspedalKarlsen} or sedimentation-consolidation processes \cite{Bustosetalbook}.
In the case $\sigma=\eta=A=0$, the equation \eqref{eq:stoc_con_brown} becomes 
a standard conservation laws in $\R^d$.
For the conservation laws, the question of existence and uniqueness of solutions  
was first settled in the pioneer papers of Kru\v{z}kov \cite{Kruzkov} and Vol'pert \cite{Volpert}. 
In the case $\sigma=\eta=0$, the equation \eqref{eq:stoc_con_brown} becomes 
a degenerate parabolic-hyperbolic equation in $\R^d$. For degenerate parabolic-hyperbolic equations entropy solution were first considered by Vol'pert and Hudajev \cite{VolpertHudajev}, while uniqueness of entropy solutions was first proved by
Carrillo \cite{carrillo_1999}. A number of authors have contributed since then, and we mention the
works of Andreianov $\&$ Maliki \cite{boris_2010}, Cockburn et al. \cite{cockburn}, Bendahmane $\&$ Karlsen \cite{mostafakarlsen_2004, mostafakarlsen_2005},
Evje et al. \cite{karl-resibro-2000} and Vallet \cite{Vallet_2005}.

\subsection{Stochastic Balance Laws}   
The study of stochastic balance laws has so far been limited mostly to equations of the type \eqref{eq:stoc_con_brown} with $A=0$. 
In fact, Kim \cite{KIm2005} extended the Kru\v{z}kov well-posedness theory to one dimensional balance laws that are driven by 
additive Brownian noise, and Vallet \& Wittbold \cite{Vallet_2009} to the multidimensional Dirichlet problem. However, when the noise 
is of multiplicative nature, one could not apply a straightforward Kru\v{z}kov's doubling method to get uniqueness. 
This issue was settled by Feng $\&$ Nualart \cite{nualart:2008}, who established uniqueness of entropy solution by 
recovering additional information from the vanishing viscosity method. The existence was proven using stochastic 
version of compensated compactness method and it was valid for \emph{one} spatial dimension.
To overcome this problem, Debussche $\&$ Vovelle \cite{Vovelle2010} introduced
kinetic formulation of such problems and as a result they were able to establish the well-posedness
of multidimensional stochastic balance law via \emph{kinetic} approach. 
A number of authors have contributed since then, and we mention the works of 
Bauzet et al. \cite{BaVaWit_2012,BaVaWit_JFA}, Biswas et al. \cite{BisMaj,BisMajKarl_2014}.
We also mention works by Chen et al. \cite{Chen-karlsen_2012}, and Biswas et al. \cite{BisKoleyMaj}, where well-posedness of entropy solution is established in $L^p \cap BV$, via BV framework. 
Moreover, they were able to develop continuous dependence theory for multidimensional balance laws and as a by product they derived an explicit \emph{convergence rate} of the approximate solutions to the underlying problem.

\subsection{Degenerate Stochastic Balance Laws}   
Stochastic degenerate parabolic-hyperbolic equations are one of the most important classes of nonlinear stochastic PDEs. Nonlinearity and degeneracy are two main features of these equations and yields several striking phenomena. In fact, due to strong degeneracy, one cannot expect smooth solutions even if the initial data is smooth. Existence and uniqueness of solutions of \eqref{eq:stoc_con_brown} was first settled by Bauzet et al. \cite{BaVaWit_2014} (in the case of $\eta=0$), and by Biswas et al. \cite{BisMajVal} (in the case of $\eta\neq0$). In fact, they have extended their previous works (\cite{BaVaWit_2012} and \cite{BisMajKarl_2014}, respectively) to the context of degenerate parabolic-hyperbolic problem in the spirit of Carrillo's work \cite{carrillo_1999}. The existence of solution is proved by using a vanishing viscosity method, based on the compactness proposed by the theory of Young measures. The uniqueness of the solution is obtained via Kruzkov's doubling variable method.
We also mention the work of Debussche et al. \cite{martina}, where the authors have established the well-posedness theory for solutions of the Cauchy problem \eqref{eq:stoc_con_brown} in any space dimension. They have adapted the notion of kinetic formulation and kinetic solution which has
already been studied in the case of hyperbolic scalar conservation laws in both
deterministic \cite{lions} and stochastic setting \cite{Vovelle2010}.

Independently of the smoothness of the initial data $u_0$, due to the presence of nonlinear flux term, degenerate diffusion term, and a nonlocal term in equation \eqref{eq:stoc_con_brown}, solutions to \eqref{eq:stoc_con_brown} are not necessarily smooth and weak solutions must be sought. Before introducing the concept of weak solutions, we first recall the notion of predictable $\sigma$-field. By a predictable $\sigma$-field on $[0,T]\times \Omega$, denoted
by $\mathcal{P}_T$, we mean that the $\sigma$-field generated by the sets of the form: $ \{0\}\times A$ and $(s,t]\times B$  for any $A \in \mathcal{F}_0; B \in \mathcal{F}_s,\,\, 0<s,t \le T$.
The notion of stochastic weak solution is defined as follows:
 
\begin{defi}[Stochastic weak solution]\label{defi:weak-solution}
 A square integrable $ L^2(\R^d )$-valued $\{\mathcal{F}_t: t\geq 0 \}$-predictable stochastic process $u(t)= u(t,x)$ is said to be a weak solution
 	to our problem \eqref{eq:stoc_con_brown} provided 
 \begin{itemize}
 \item [(i)]  $u \in L^2(\Omega \times \Pi_T)$ and $A(u)\in L^2((0,T)\times \Omega; H^1(\R^d))$.
 \item[(ii)] $\frac{\partial}{\partial t} \Big[u-\int_0^t \sigma(u(s,\cdot))\,dW(s)- \int_0^t \int_{|z|>0}\eta(u(s,\cdot);z)\,\widetilde{N}(dz,ds)\Big]
 \in L^2((0,T)\times \Omega; H^{-1}(\R^d))$ in the sense of distribution.
 \item[(iii)] For almost every $t\in [0,T]$ and $ P-$ a.s, the following variational formulation holds: 
 \begin{align}
 &\Big\langle \frac{\partial}{\partial t} \Big[u-\int_0^t  \sigma(u(s,\cdot)) \,dW(s)-\int_0^t \int_{|z|>0}\eta(u(s,\cdot);z)\,\widetilde{N}(dz,ds)\Big],v \Big\rangle_{\big({H^{-1}(\R^d),H^1(\R^d)} \big)}\notag \\
 & \qquad \quad + \int_{\R^d} \Big\{ \grad A(u(t,x)) + f(u(t,x))\Big\}.\grad v\,dx=0,
 \end{align}
for any $v \in H^1(\R^d)$. 
 \end{itemize}
 \end{defi}


However, it is well-known that weak solutions may be discontinuous and they are not uniquely determined by their initial data. Consequently, an
admissible condition so called {\em entropy solution}  must be imposed to single out the physically correct solution. Since the notion of
entropy solution is built around the so called entropy flux triple, we begin with the definition of entropy flux triple.

\begin{defi}[Entropy flux triple]
 	A triplet $(\beta,\zeta,\nu) $ is called an entropy flux triple if $\beta \in C^2(\R) $ and $\beta \ge0$,
 	$\zeta = (\zeta_1,\zeta_2,....\zeta_d):\R \mapsto \R^d$ is a vector valued function, and $ \nu :\R \mapsto \R $ is a scalar valued function such that 
 	\[\zeta'(r) = \beta'(r)f'(r) \quad \text{and}\quad \nu^\prime(r)= \beta'(r)A'(r).\]
 An entropy flux triple $(\beta,\zeta,\nu)$ is called convex if $ \beta^{\prime\prime}(s) \ge 0$.  
\end{defi}
To define entropy solution, we first define associated Kirchoff's function of $A$, denoted by 
$G(x)$ as $G(x)= \int_0^x \sqrt{A^\prime(r)}\,dr$. With the help of a convex entropy flux triple $(\beta,\zeta,\nu)$, the notion of stochastic entropy solution is defined as follows:

 \begin{defi} [Stochastic entropy solution]
 \label{defi:stochentropsol}
A square integrable $ L^2(\R^d )$-valued $\{\mathcal{F}_t: t\geq 0 \}$-predictable stochastic process $u(t)= u(t,x)$ is called a stochastic entropy solution of \eqref{eq:stoc_con_brown} if
 \begin{itemize}
 \item[(i)] For each $ T>0$, 
 \begin{align*}
 G(u) \in L^2((0,T)\times \Omega;H^1(\R^d)), \,\, \text{and} \,\, 
 \underset{0\leq t\leq T}\sup  \E\big[||u(t,\cdot)||_{2}^{2}\big] <+ \infty. 
 \end{align*}
 \item[(ii)] Given a non-negative test function  $\psi\in C_{c}^{1,2}([0,\infty )\times\R^d) $ and a convex entropy flux triple $(\beta,\zeta,\nu)$, the following inequality holds:
 \begin{align}
 &  \int_{\Pi_T} \Big\{ \beta(u(t,x)) \partial_t\psi(t,x) +  \nu(u(t,x))\Delta \psi(t,x) -  \grad \psi(t,x)\cdot \zeta(u(t,x)) \Big\}dx\,dt \notag \\
 & + \int_{\Pi_T} \sigma(u(t,x))\beta^\prime (u(t,x))\psi(t,x)\,dW(t)\,dx
 + \frac{1}{2}\int_{\Pi_T}\sigma^2(u(t,x))\beta^{\prime\prime} (u(t,x))\psi(t,x)\,dx\,dt \notag \\
  &  + \int_{\Pi_T} \int_{|z|>0} \int_0^1 \eta(u(t,x);z)\beta^\prime \big(u(t,x) + \lambda\,\eta(u(t,x);z)\big)\psi(t,x)\,d\lambda\,\widetilde{N}(dz,dt)\,dx  \notag \\
 &\quad +\int_{\Pi_T} \int_{|z|>0}  \int_0^1  (1-\lambda)\eta^2(u(t,x);z)\beta^{\prime\prime} \big(u(t,x) + \lambda\,\eta(u(t,x);z)\big)
 \psi(t,x)\,d\lambda\,m(dz)\,dx\,dt \notag \\
 &  \quad \ge  \int_{\Pi_T} \beta^{\prime\prime}(u(t,x)) |\grad G(u(t,x))|^2\psi(t,x)\,dx\,dt - \int_{\R^d} \beta(u_0(x))\psi(0,x)\,dx, \quad \mathbb{P}-\text{a.s}.\label{inq:entropy-solun}
 \end{align}
 \end{itemize}
 \end{defi} 
 
Due to nonlocal nature of the  It\^{o}-L\'{e}vy  formula  and the noise-noise interaction, 
the Definition~\ref{defi:stochentropsol} alone does not give the $L^1$-contraction principle in the sense of average and hence the uniqueness.
In fact, classical ``doubling of variable'' technique in time variable does not work when one tries to compare directly two entropy 
solutions defined in the sense of Definition~\ref{defi:stochentropsol}. To overcome this problem, the authors 
in \cite{BaVaWit_2014, BisMajVal} used a more direct approach by comparing solutions of two regularized problems and subsequently 
sending the regularized parameter to zero, relying on ``weak compactness'' of the regularized approximations.

In order to successfully implement the direct approach, one needs to weaken the 
notion of stochastic entropy solution, and subsequently install the notion of so called generalized entropy solution.
\begin{defi} [Generalized Entropy Solution]
\label{defi: young_stochentropsol}
A square integrable $ L^2\big(\R \times (0,1)\big)$-valued $\{\mathcal{F}_t: t\geq 0 \}$-predictable stochastic
process $u(t)= u(t,x,\alpha)$ is called a generalized stochastic entropy solution of \eqref{eq:stoc_con_brown} if 
	\begin{itemize}
 \item[(i)] For each $ T>0$, 
 \begin{align*}
 G(u) \in L^2((0,T)\times \Omega \times(0,1);H^1(\R^d)), \,\, \text{and} \,\, 
 \underset{0\leq t\leq T}\sup  \E\big[||u(t,\cdot, \cdot)||_{2}^{2}\big] < \infty. 
 \end{align*}
 \item[(ii)] Given a non-negative test function  $\psi\in C_{c}^{1,2}([0,\infty )\times\R^d) $ and a convex entropy flux triple $(\beta,\zeta,\nu)$, the following inequality holds:
 \begin{align}
 &  \int_{\Pi_T} \int_0^1\Big\{ \beta(u(t,x, \alpha)) \partial_t\psi(t,x) +  \nu(u(t,x, \alpha))\Delta \psi(t,x) -  \grad \psi(t,x)\cdot \zeta(u(t,x, \alpha)) \Big\}\, d\alpha \,dx\,dt \notag \\
 & \quad + \int_{\Pi_T} \int_0^1 \sigma(u(t,x, \alpha))\beta^\prime (u(t,x, \alpha))\psi(t,x)\,d\alpha \,dW(t)\,dx \notag \\
 &\quad + \frac{1}{2}\int_{\Pi_T}\int_0^1\sigma^2(u(t,x, \alpha))\beta^{\prime\prime} (u(t,x, \alpha))\psi(t,x)\,d\alpha\,dx\,dt \notag \\
  &  \quad + \int_{\Pi_T}\int_0^1 \int_{|z|>0} \int_0^1 \eta(u(t,x, \alpha);z)\beta^\prime \big(u(t,x, \alpha) + \lambda\,\eta(u(t,x,\alpha);z)\big)\psi(t,x)\,d\alpha\,d\lambda\,\widetilde{N}(dz,dt)\,dx  \notag \\
 &\quad +\int_{\Pi_T}\int_0^1 \int_{|z|>0}  \int_0^1  (1-\lambda)\eta^2(u(t,x, \alpha);z)\beta^{\prime\prime} \big(u(t,x, \alpha) + \lambda\,\eta(u(t,x, \alpha);z)\big)
 \psi(t,x)\,d\alpha\,d\lambda\,m(dz)\,dx\,dt \notag \\
 &  \quad \ge  \int_{\Pi_T}\int_0^1 \beta^{\prime\prime}(u(t,x, \alpha)) |\grad G(u(t,x, \alpha))|^2\psi(t,x)\,d\alpha\,dx\,dt
 - \int_{\R^d} \beta(u_0(x))\psi(0,x)\,dx, \quad \mathbb{P}-\text{a.s}.\label{inq:entropy-solun_1}
 \end{align}
\end{itemize}
\end{defi}
As we mentioned earlier, in \cite {BaVaWit_2014,BisMajVal},
the authors have revisited \cite{boris_2010,carrillo_1999,chenkarlsen_2005} and established the well posedness of the entropy solution in the sense of Definition \ref{defi:stochentropsol} \textit{via} Young's measure theory.

\subsection{Scope and Outline of the Paper}
The above discussions clearly highlight the lack of stability estimates for the entropy solutions of degenerate 
parabolic-hyperbolic stochastic balance laws driven by L\'{e}vy noise. In this paper, drawing preliminary motivation from \cite{BisKoleyMaj,Chen-karlsen_2012,karl-resibro-2000},
we intend to develop a continuous dependence theory for stochastic entropy solution which in turn can be used to derive an error estimate for the vanishing viscosity method. However, it seems difficult to develop such a theory without securing a BV estimate for stochastic entropy solution. As a result, we first address the question of existence, uniqueness of stochastic BV entropy solution in $L^2(\R^d) \cap BV(\R^d)$ of the problem \eqref{eq:stoc_con_brown}. Making use of the crucial BV estimate, we provide a continuous dependence estimate and error estimate
for the vanishing viscosity method provided initial data lies in $u_0 \in L^2(\R^d) \cap BV(\R^d)$.
Finally, we turn our discussions to more general degenerate stochastic balance laws driven by L\'{e}vy processes,
namely when the functions $ \sigma, \eta$ appear in the L\'{e}vy noise has explicit dependency on the spatial position $x$ as well. In view of the discussions in \cite{BisKoleyMaj,Chen-karlsen_2012}, in this case we can't expect BV estimates, but instead a fractional BV estimate is expected. However, that does not prevent us to provide an existence proof for more general class of equations in $L^2(\R^d)$.
 
The rest of the paper is organized as follows:  we describe technical framework and state the main results in Section~\ref{sec:tech}. In Section~\ref{sec:apriori+existence}, we derive uniform spatial BV bound for viscous solutions. Using this bound, we establish well posedness of BV entropy solution of the Cauchy problem \eqref{eq:stoc_con_brown}. Section~\ref{sec:main-thm} is devoted on deriving the continuous dependence estimate on nonlinearities, while Section~ \ref{sec:cor} deals with the error estimates. Finally, in Section~\ref{sec:frac}, we establish a fractional BV estimate for a larger class of degenerate stochastic balance laws.
  
 
\section{Technical Framework and Statement of the Main Results}
\label{sec:tech}
Throughout this paper, we use the letter $C$ to denote various generic constants. There are situations where constant may change from line to line, but the notation is kept unchanged so long as it does not impact central idea. 
Moreover, for any separable Hilbert space $H$, we denote by $N_w^2(0,T,H)$, the Hilbert space of all the predictable $H$-valued processes $u$ such that $\E\Big[\int_0^T \|u\|^2_H\Big]<+\infty$.
Furthermore, we denote $BV(\R^d)$ as the set of integrable 
 functions with bounded variation on $\R^d$ endowed with the norm $|u|_{BV(\R^d)}= \|u\|_{L^1(\R^d)} + TV_{x}(u)$, where $TV_{x}$ is the total variation of $u$ defined on $\R^d$. 
The primary aim of this paper is to derive continuous dependence estimates for the entropy solutions of the Cauchy problem \eqref{eq:stoc_con_brown}, and we do so under the following assumptions:
 \begin{Assumptions}
 	\item \label{A1} The initial function $u_0$ is a $\mathcal{F}_0$ measurable random variable satisfying 
 	\begin{align*}
 	\E\Big[ \|u_0\|_2^2 + |u_0|_{BV(\R^d)}\Big] < + \infty.
 	\end{align*}
 	\item \label{A2}  $ A:\R\rightarrow \R$ is a non-decreasing  Lipschitz continuous function with $A(0)=0$. Furthermore, we assume that $A^{\prime\prime}$ is bounded.
 	\item \label{A3}  $ f=(f_1,f_2,\cdots, f_d):\R\rightarrow \R^d$ is a Lipschitz continuous function with $f_k(0)=0$, for 
 	all $1\le k\le d$.
 	\item \label{A4} We assume that $\sigma(0)=0$. Moreover, there exists positive constant $K > 0$  such that 
 	\begin{align*} \big| \sigma(u)-\sigma(v)\big|  \leq K |u-v|, ~\text{for all} \,\,u,v \in \R.
 	\end{align*}  
\item \label{A5} There exist positive constants $\lambda^* \in (0,1)$, and $C>0$ such that for all $u,v \in \R;~~z\in \R$
 \begin{align*}
  \big| \eta(u;z)-\eta(v;z)\big|  \leq \lambda^* |u-v|( |z|\wedge 1),\, \,\,\text{and}\,\,\, |\eta(u,z)|\le C\,|u|( |z|\wedge 1).
 \end{align*}
 Moreover, we assume that $\eta(0,z)=0$, for all $z\in \R$. 
 \item \label{A6} The L\'{e}vy measure $m(dz)$ is a Radon measure on $\R\backslash \{0\}$ with a possible singularity at $z=0$,
 which satisfies
\begin{align*}
    \int_{|z|>0}(1 \wedge |z|^2)\, m(dz) < +\infty. 
\end{align*}
 \end{Assumptions}

\begin{rem}
We remark that, one can accommodate polynomially growing flux function as a result of the requirement that the entropy solutions satisfy $L^p$ bounds for all $p\ge 2$. This in turn forces to choose initial data that are in $L^p$, for all $p$. However, we have chosen to work with the assumptions ~\ref{A1} and ~\ref{A3}. The assumption ~\ref{A5} is natural in the context of L\'{e}vy noise with the exception of $\lambda^*\in (0,1)$, which is necessary for the uniqueness.
Finally, the assumptions \ref{A1}-\ref{A6} collectively ensures existence and uniqueness of stochastic entropy solution,  and the continuous dependence estimate as well. 
\end{rem}

Like its deterministic counterpart, existence of entropy solution largely related to the study of associated viscous problem. To this end, for $\eps>0$, we consider viscous approximation of \eqref{eq:stoc_con_brown} as
 \begin{align}
 	 du_\eps(t,x) -\Delta A_\eps(u_\eps(t,x))\,dt - \mbox{div}_x f(u_\eps(t,x)) \,dt&= \sigma(u_\eps(t,x))\,dW(t) + \int_{|z|>0} \eta(u_\eps(t,x);z)\widetilde{N}(
 	 dz,dt),\label{eq:viscous-Brown} \\
 	 u_\eps(0,x)&=u_0(x), \notag
\end{align}
where $A_\eps(x)=A(x)+ \eps x$. One can follow the argument presented in \cite{BaVaWit_2014,BisMajVal,Vallet_2008} to ensure existence of weak solutions for the problem \eqref{eq:viscous-Brown}. More precisely, we have the following proposition from \cite{BaVaWit_2014,BisMajVal,Vallet_2008}.

\begin{prop}\label{prop:vanishing viscosity-solution}
Let the assumptions \ref{A1},\,\ref{A3},\,\ref{A4},\,\ref{A5} and \ref{A6} hold and $A:\R\mapsto \R$ is non-decreasing Lipschitz continuous function. Then, for any $\eps>0$, there exists a unique weak solution $u_\eps \in N_w^2(0,T,H^1(\R^d))$  with $\partial_t \big(u_\eps - \int_0^t \sigma(u_\eps(s,\cdot))\, dW(s) - \int_0^t \int_{|z|>0} \eta(u_\eps(s,\cdot);z)\widetilde{N}(dz,ds)\big)
\in L^2(\Omega\times(0,T),H^{-1}(\R^d))$,  to the problem \eqref{eq:viscous-Brown}. Moreover,
$u_\eps \in L^\infty(0,T;L^2(\Omega \times\R^d))$ and there exists a constant $C>0$, independent of $\eps$, such that
\begin{align}
	 \sup_{0\le t\le T} \E\Big[\big\|u_\eps(t)\big\|_2^2\Big]  + \eps \int_0^T \E\Big[\big\|\grad u_\eps(s)\big\|_2^2\Big]\,ds
	 + \int_0^T \E\Big[\big\|\grad G(u_\eps(s))\big\|_2^2\Big]\,ds \le C,\label{bounds:a-priori-viscous-solution}
\end{align} 
where $G$ is the associated Kirchoff's function of $A$. 
\end{prop}

We are now in a position to state the main results of this paper. 
\begin{maintheorem*} [Continuous dependence estimate]
 Let the assumptions \ref{A1}-\ref{A6} hold for two sets of given data $(u_0, f, A,\sigma, \eta)$ and $(v_0,g,B,\widetilde{\sigma}, \widetilde{\eta})$. 
 Let $u(t,x)$ be any BV entropy solution of \eqref{eq:stoc_con_brown} with initial data $u_0(x)$ and $v(s,y)$ be another BV entropy solution with initial
 data $v_0(x)$ and satisfies 
\begin{align}
dv(s,y) -\Delta  B(v(s,y))\,ds - \mbox{div} g(v(s,y)) \,ds =\widetilde{\sigma}(v(s,y))\,dW(s) + \int_{|z|>0} \widetilde{\eta}(v(s,y);z)\,\widetilde{N}(dz,ds).\label{eq:stoc_con_brown-1}
\end{align}
Moreover, define
\begin{align*}
 \mathcal{E}(\sigma, \widetilde{\sigma}):=& \sup_{\xi \neq0} \frac{|\sigma(\xi)-\widetilde{\sigma}(\xi)|}{|\xi|}, \\
 \mathcal{D}(\eta,\widetilde{\eta}):=& \displaystyle{ \sup_{u \neq0} \int_{|z|>0} \frac{\big| \eta(u;z)-\widetilde{\eta}(u;z)\big|^2
}{|u|^2}\, m(dz)},
\end{align*}
and, in addition, assume that $f^{\prime\prime}\in L^\infty$.
Then, there exists a constant $C_T$, depending on  $T$, $ |u_0|_{BV(\R^d)}$, $|v_0|_{BV(\R^d)}$, $\|f^{\prime\prime}\|_{\infty}$, 
$\|f^\prime\|_{\infty}$, $\|\Phi\|_{1}$ and $\|B^\prime\|_{\infty}$ such that for a.e. $0<t<T<+\infty$, 
  \begin{align*}
     & \E \Big[\int_{\R^d}\big| u(t,x)-v(t,x)\big|\Phi(x)\,dx \Big] \notag \\
     & \hspace{1cm} \le C_Te^{Ct} \Bigg\{ \E\Big[\int_{\R^d}\big| u_0(x) -v_0(x)\big| \,dx\Big] 
     + \max \bigg\{ \mathcal{E}(\sigma, \widetilde{\sigma}),\sqrt{\mathcal{D}(\eta,\widetilde{\eta})}\bigg\} \sqrt{t} + ||f^\prime-g^\prime||_{\infty} t  \\
      & \hspace{5cm}+ \max\bigg\{ \sqrt{\|A'-B'\|_\infty}, \, \sqrt{\mathcal{E}(\sigma, \widetilde{\sigma})}, \sqrt[4]{\mathcal{D}(\eta,\widetilde{\eta})}  \bigg\}\sqrt{t} \Bigg\},
\end{align*}
   where $\Phi \in L^1(\R^d)$ such that $ 0\le \Phi(x)\le 1$, for all $x\in \R^d$. 
\end{maintheorem*}

As a by product of the above theorem, we have the following corollary:
\begin{maincorollary*}[Rate of convergence]
Let the assumptions \ref{A1}-\ref{A6} hold and $f^{\prime\prime}\in L^\infty$. Let $u(t,x)$ be any BV entropy solution of \eqref{eq:stoc_con_brown}
with $\E\Big[|u(t,\cdot)|_{BV(\R^d)}\Big] \le \E\Big[|u_0(\cdot)|_{BV(\R^d)}\Big]$ and 
  $u_\eps(s,y)$ be a weak solution to the problem \eqref{eq:viscous-Brown}. 
  Then there exists a constant $C$ depending only on 
  $|u_0|_{BV(\R^d)}, \|f^{\prime\prime}\|_{\infty}, \|f^\prime\|_{\infty}$, and $\|A^\prime\|_{\infty}$ such that for a.e. $t>0$, 
\begin{align*}
 \E\Big[\|u_\eps(t,\cdot)-u(t,\cdot)\|_{L^1(\R^d)}\Big] \le Ce^{Ct} \,\eps^\frac{1}{2}.
\end{align*}	
\end{maincorollary*}

Before concluding this section, we introduce a  special class of entropy functions, 
called convex approximation of absolute value function. To do so,  let $\beta:\R \rightarrow \R$ be a $C^\infty$ function satisfying 
\begin{align*}
\beta(0) = 0,\quad \beta(-r)= \beta(r),\quad 
\beta^\prime(-r) = -\beta^\prime(r),\quad \beta^{\prime\prime} \ge 0,
\end{align*} 
and 
\begin{align*}
\beta^\prime(r)=\begin{cases} -1\quad \text{when} ~ r\le -1,\\
\in [-1,1] \quad\text{when}~ |r|<1,\\
+1 \quad \text{when} ~ r\ge 1.
\end{cases}
\end{align*} 
For any $\xi> 0$, define $\beta_\xi:\R \rightarrow \R$ by 
$\beta_\xi(r) = \xi \beta(\frac{r}{\xi})$. 
Then
\begin{align}\label{eq:approx to abosx}
|r|-M_1\xi \le \beta_\xi(r) \le |r|\quad 
\text{and} \quad |\beta_\xi^{\prime\prime}(r)| 
\le \frac{M_2}{\xi} {\bf 1}_{|r|\le \xi},
\end{align} 
where $M_1 := \sup_{|r|\le 1}\big | |r|-\beta(r)\big |$ and 
$M_2 := \sup_{|r|\le 1}|\beta^{\prime\prime} (r)|$.

\begin{rem}\label{beta}
Note that if $\beta_\xi$ is an even, non-negative, convex function and if  $\beta^{\prime\prime}_\xi$ is non-increasing on the positive reals, then, for any positive $r$,
\begin{align}
&C_\xi r^2 \geq 2\beta_\xi(r) = 2\int_0^r\int_0^s \beta^{\prime\prime}_\xi(\sigma)d\sigma ds \geq r^2 \beta^{\prime\prime}_\xi(r) \label{beta_1}
\\ \text{ and }&\ \forall \alpha \geq 1,  \nonumber
\\&
\beta_\xi(\alpha r) =\alpha^2 \int_0^r\int_0^s \beta^{\prime\prime}_\xi(\alpha\sigma)d\sigma ds \leq \alpha^2 \beta_\xi(r).  \label{beta_2}
\end{align}
\end{rem}

Moreover, for $\beta=\beta_\xi$, we define 
\begin{align*}
\begin{cases}
f_k^\beta(a,b)= \int_b^a \beta_\xi ^\prime(r-b) f_k^\prime(r)\,dr, \\
f^\beta(a,b)=\big(f_1^\beta(a,b),f_2^\beta(a,b),\cdots,f_d^\beta(a,b)\big), \\
A^\beta(a,b)= \int_b^a \beta_\xi ^\prime(r-b) A^\prime(r)\,dr.
\end{cases}
\end{align*}

\section{A Priori Estimates}
\label{sec:apriori+existence}
In this section, we derive uniform spatial BV bound for the solutions of degenerate parabolic-hyperbolic stochastic balance laws driven by L\'{e}vy noise \eqref{eq:stoc_con_brown} under the assumptions $\ref{A1}$-$\ref{A6}$. Like its deterministic counter part, we first secure uniform spatial BV bound for the viscous solutions, i.e., solutions of \eqref{eq:viscous-Brown}. Regarding this, we have the following theorem. 
\begin{thm}
\label{thm:bv-viscous}
Let the assumptions $\ref{A1}$-$\ref{A6}$ hold. For $\eps>0$, let $u_\eps(t,x)$  be a solution to the Cauchy problem \eqref{eq:viscous-Brown}. Then there exists a constant $C>0$, independent of $\eps$, such that for any time $t>0$, 
\begin{align*}
\sup_{\eps>0} \E\Big[ \|u_\eps(t)\|_{L^1(\R^d)}\Big] \le  C\E\Big[\|u_0\|_{L^1(\R^d)} \Big], \quad
&\sup_{\eps>0} \E\Big[ TV_x(u_\eps(t))\Big] \le  \E\Big[TV_x(u_0) \Big].
 \end{align*}
\end{thm}

\begin{rem}
In view of the lower semi-continuity property and the positivity of the total variation $TV_x$, we point out that $u \mapsto \E[TV_x(u)]$ makes sense for 
 any $u\in L^1(\Omega \times \R^d)$ as a real-extended lsc convex function.
\end{rem}

\begin{proof}
For a proof of the first part of the above theorem, consult Appendix~\ref{appendix}. For the second part, we proceed as follows: 
Set $\eps>0$ and let $u_\eps$ be the weak solution to the problem  \eqref{eq:viscous-Brown} and $v_\eps$ be a weak solution to the stochastic equation 
\begin{align}
dv_\eps(t,x) -\Delta A_\eps(v_\eps(t,x))\,dt - \mbox{div}_x f(v_\eps(t,x)) \,dt&=\sigma(v_\eps(t,x))\,dW(t) +  \int_{|z|>0} \eta(v_\eps(t,x);z)\widetilde{N}(dz,dt),\notag \\
  v_\eps(0,x)&=v_0(x),
\end{align}
Then, it is evident that $u_\eps -v_\eps$ is a stochastic weak solution to the problem 
\begin{align*}
 \begin{cases}
  d(u_\eps(t,x)-v_{\eps}(t,x))- \Delta \big( A_\eps(u_\eps(t,x))-A_\eps(v_\eps(t,x))\,dt - \mbox{div}_x\big( f(u_\eps(t,x))- f(v_\eps(t,x))\big)\,dt\\
  \hspace{1.5cm} = \big(\sigma(u_\eps(t,x))-\sigma(v_\eps(t,x))\Big)\,dW(t) + \int_{|z|>0}\big(\eta(u_\eps(t,x);z)-\eta(v_\eps(t,x);z)\big)\,\widetilde{N}(dz,dt),\\
  u_\eps -v_\eps\big|_{(t=0,x)} = u_0(x)-v_0(x).
 \end{cases}
 \end{align*}
Note that $u_\epsilon-v_\epsilon$ is a weak solution and not a strong one. Thus, we apply a slight modification 
of It\^{o}-L\'{e}vy formula (as proposed in Fellah \cite{fellah} and Biswas et al. \cite{BisMajVal}) to $\int_{\R^d}\beta_\xi(u_\eps-v_\eps)dx$, 
where $\beta_\xi$ is defined in Section \ref{sec:tech} and then take expectation. The result is
\begin{align}
& \E\Big[ \int_{\R^d} \beta_\xi \big(u_\eps(t,x)-v_\eps(t,x)\big)\,dx\Big] \notag \\
&= \E\Big[ \int_{\R^d} \beta_\xi \big(u_0(x)-v_0(x)\big)\,dx\Big] \notag \\
& - \E\Big[ \int_{\R^d}
\int_0^t \beta_\xi^{\prime\prime} \big( u_\eps(s,x)-v_\eps(s,x)\big) \grad \big(A_{\eps}(u_\eps(s,x))-A_\eps(v_\eps(s,x))\big)\cdot \grad \big(u_\eps(s,x)-v_\eps(s,x)\big)  \,ds\,dx\Big] \notag \\
&\quad - \E\Big[ \int_{\R^d}
\int_0^t \beta_\xi^{\prime\prime} \big( u_\eps(s,x)-v_\eps(s,x)\big) \big(f(u_\eps(s,x))-f(v_\eps(s,x))\big)\cdot \grad \big(u_\eps(s,x)-v_\eps(s,x)\big)  \,ds\,dx\Big] \notag \\
& \qquad + \frac{1}{2}  \E\Big[ \int_{\R^d}
\int_0^t \beta_\xi^{\prime\prime} \big( u_\eps(s,x)-v_\eps(s,x)\big) \big(\sigma(u_\eps(s,x))-\sigma(v_\eps(s,x))\big)^2 \,ds\,dx\Big] \notag \\
&\quad \quad  + \E\Big[ \int_{\R^d}
\int_0^t \int_{|z|>0} \int_0^1 (1-\lambda) \beta_\xi^{\prime\prime} \Big( u_\eps(s,x)-v_\eps(s,x) + \lambda \big(\eta(u_\eps(s,x);z)-\eta(v_\eps(s,x);z)\big)\Big)\notag \\
 & \hspace{6cm} \times \big(\eta(u_\eps(s,x);z)-\eta(v_\eps(s,x);z)\big)^2 \,d\lambda\,m(dz)\,ds\,dx\Big] \notag \\
& := \mathcal{A} + \mathcal{B}+ \mathcal{C} + \mathcal{D} +\mathcal{G}.
\end{align}
Our aim is to estimate each of the above terms separately. Let us first consider the term $\mathcal{B}$. Note that, since $-A_\eps^\prime(x)\le -\eps$, we have 
\begin{align}
\mathcal{B}&= - \E\Big[ \int_{\R^d}
\int_0^t \beta_\xi^{\prime\prime} \big( u_\eps(s,x)-v_\eps(s,x)\big) A_\eps^\prime(u_\eps(s,x)) \Big|\grad \big(u_\eps(s,x)-v_\eps(s,x)\big)\Big|^2  \,ds\,dx\Big] \notag \\
& - \E\Big[ \int_{\R^d}
 \int_0^t \beta_\xi^{\prime\prime} \big( u_\eps(s,x)-v_\eps(s,x)\big)  \big(A_{\eps}^\prime(u_\eps(s,x))-A_\eps^\prime(v_\eps(s,x))\big) \grad v_\eps(s,y)\cdot \grad \big(u_\eps(s,x)-v_\eps(s,x)\big)  \,ds\,dx\Big] \notag \\
&\le -  \eps \, \E\Big[ \int_{\R^d}
\int_0^t \beta_\xi^{\prime\prime} \big( u_\eps(s,x)-v_\eps(s,x)\big) \Big|\grad \big(u_\eps(s,x)-v_\eps(s,x)\big)\Big|^2  \,ds\,dx\Big] \notag \\
& - \E\Big[ \int_{\R^d}
\int_0^t \beta_\xi^{\prime\prime} \big( u_\eps(s,x)-v_\eps(s,x)\big)  \big(A_{\eps}^\prime(u_\eps(s,x))-A_\eps^\prime(v_\eps(s,x))\big) \grad v_\eps(s,y)\cdot \grad \big(u_\eps(s,x)-v_\eps(s,x)\big)  \,ds\,dx\Big] \notag \\
& := \mathcal{B}_1 + \mathcal{B}_2.\label{eq:b}
\end{align}
Now consider the term $\mathcal{B}_2$. Thanks to the Young's inequality, we obtain 
\begin{align}
\big|\mathcal{B}_2\big| \le & \frac{\eps}{4} \, \E\Big[ \int_{\R^d}
\int_0^t \beta_\xi^{\prime\prime} \big( u_\eps(s,x)-v_\eps(s,x)\big) \Big|\grad \big(u_\eps(s,x)-v_\eps(s,x)\big)\Big|^2  \,ds\,dx\Big] \notag \\ 
& + C(\eps) \, \E\Big[ \int_{\R^d}
\int_0^t \beta_\xi^{\prime\prime} \big( u_\eps(s,x)-v_\eps(s,x)\big)  \big(A_{\eps}^\prime(u_\eps(s,x))-A_\eps^\prime(v_\eps(s,x))\big)^2 \big|\grad v_\eps(s,x)\big|^2 \,ds\,dx\Big].\label{esti:b2}
\end{align}
Thus, combining \eqref{esti:b2} and \eqref{eq:b}, we get 
\begin{align}
\mathcal{B} \le & - \frac{3 \eps}{4} \E\Big[ \int_{\R^d}
\int_0^t \beta_\xi^{\prime\prime} \big( u_\eps(s,x)-v_\eps(s,x)\big) \Big|\grad \big(u_\eps(s,x)-v_\eps(s,x)\big)\Big|^2  \,ds\,dx\Big] \notag \\ 
& + C(\eps) \, \E\Big[ \int_{\R^d}
\int_0^t \beta_\xi^{\prime\prime} \big( u_\eps(s,x)-v_\eps(s,x)\big)
\big(A^\prime(u_\eps(s,x))-A^\prime(v_\eps(s,x))\big)^2 \big|\grad v_\eps(s,x)\big|^2 \,ds\,dx\Big] \notag \\
&:= \mathcal{B}_3 + \mathcal{B}_4.\label{esti:b}
\end{align}
 Let us focus on the term $\mathcal{B}_4$. Note that, in view of last part of the assumption \ref{A2}, $A^{\prime\prime}$ is bounded. 
Using that along with the estimate \eqref{bounds:a-priori-viscous-solution} and the fact that 
 $r^2 \beta_{\xi}^{\prime\prime}(r) \le C \xi$ for any $r\in \R$,  we estimate $\mathcal{B}_4$ as 
\begin{align}
 \mathcal{B}_4& \le C(\eps) \xi \E\Big[\int_0^t \int_{\R^d} \big|\grad v_\eps(s,x)\big|^2\,dx\,ds \Big] \notag \\
 & \le \xi C(\eps) \int_0^T \E\Big[ \|\grad v_\eps(s)\|_2^2 \Big] ds
\le C(\eps) \,\xi \mapsto 0, \,\, \text{as}\,\, \xi \mapsto 0 \,(\text{keeping}\, \eps>0\,\,\text{fixed} ).\notag
\end{align}
Next we move on to estimate the flux term $\mathcal{C}$. In view of the Young's inequality, one has
\begin{align}
\mathcal{C} \le& \frac{\eps}{4}  \E\Big[ \int_{\R^d}
\int_0^t \beta_\xi^{\prime\prime} \big( u_\eps(s,x)-v_\eps(s,x)\big) \Big|\grad \big(u_\eps(s,x)-v_\eps(s,x)\big)\Big|^2  \,ds\,dx\Big] \notag \\ 
& \qquad \qquad + C(\eps) \, \E\Big[ \int_{\R^d}
\int_0^t \beta_\xi^{\prime\prime} \big( u_\eps(s,x)-v_\eps(s,x)\big) \big| f(u_\eps(s,x))-f(v_\eps(s,x))\big|^2  \,ds\,dx\Big] \notag \\ 
&:= \mathcal{C}_1 + \mathcal{C}_2. \notag
\end{align}
In view of the Lipschitz continuity of $f$ and \eqref{eq:approx to abosx}, we see that 
\begin{align*}
 \beta_\xi^{\prime\prime} \big( u_\eps(s,x)-v_\eps(s,x)\big) \big| f(u_\eps(s,x))-f(v_\eps(s,x))\big|^2 &  \le C \big|u_\eps(s,x)-v_\eps(s,x)\big|
 \textbf{1}_{\big\{ 0<\big| u_\eps(s,x)-v_\eps(s,x)\big| <\xi \big\}} \notag \\
 & \le C \big| u_\eps(s,x)-v_\eps(s,x)\big| \in L^1(\Omega \times (0,T)\times \R^d).
\end{align*}
On the other hand, $\big|u_\eps(s,x)-v_\eps(s,x)\big|
 \textbf{1}_{\big\{ 0<\big| u_\eps(s,x)-v_\eps(s,x)\big| <\xi \big\}} \longrightarrow 0$ as $\xi \goto 0$ for almost every $(t,x)$ and almost surely. Thus by dominated convergence theorem, we conclude that 
$\mathcal{C}_2 \le \eps(\xi)$ where $\eps(\xi) \goto 0$ as $\xi \goto 0$. A similar calculation reveals that
 $\mathcal{D}\le \eps(\xi)$  where $\eps(\xi) \goto 0$ as $\xi \goto 0$.
 \vspace{.2cm}
 
Now we move on to estimate $\mathcal{G}$. Let $a= u_\eps(s,x)-v_\eps(s,x)$ and $b=\eta(u_\eps(s,x);z)-\eta(v_\eps(s,x);z)$. Then we have, in view of assumption 
\ref{A5}, 
\begin{align}
 \mathcal{G}= &\E\Big[ \int_{\R^d}\int_0^t \int_{|z|>0} \int_0^1 (1-\lambda) b^2 \beta_\xi^{\prime\prime} \big(a + \lambda\,b\big) \,d\lambda\,m(dz)\,ds\,dx\Big] \notag \\
 & \le  \E\Big[ \int_{\R^d}\int_0^t \int_{|z|>0} \int_0^1 (1-\lambda) a^2 \beta_\xi^{\prime\prime} \big(a + \lambda\,b\big)
 (1\wedge |z|^2) \,d\lambda\,m(dz)\,ds\,dx\Big]. \label{eq:nonlocal-estim-0}
\end{align}
Note that $\beta^{\prime\prime}$ is nonnegative and symmetric around zero. 
Thus, we can assume without loss of generality that $a \ge 0$. 
Then, by assumption \ref{A5},
\begin{align*}
u_\eps(s,x)-v_\eps(s,x) +\lambda \,b \ge (1- \lambda^*) \big( u_\eps(s,x)-v_\eps(s,x)\big)
\end{align*}
for $\lambda \in [0,1]$. 
In other words    
\begin{align}
0 \le a \le (1-\lambda^*)^{-1}(a+ \lambda b). 
\label{eq:nonlocal-estim-1}
\end{align}
We combine \eqref{eq:nonlocal-estim-0} and \eqref{eq:nonlocal-estim-1} to obtain 
\begin{align*}
 \mathcal{G} \le C(\lambda^*)  \E\Big[ \int_{\R^d}\int_0^t \int_{|z|>0} \int_0^1 (1-\lambda) (a+ \lambda b)^2 
 \beta_\xi^{\prime\prime} \big(a + \lambda\,b\big) (1\wedge |z|^2) \,d\lambda\,m(dz)\,ds\,dx\Big].
\end{align*}
In view of \eqref{eq:approx to abosx}, and the assumption on $\eta$ that $\eta(0,z)=0$ for all $z\in \R$,  we see that, for each $\lambda \in [0,1]$
\begin{align*}
 (a+ \lambda b)^2 \beta_\xi^{\prime\prime} \big(a + \lambda\,b\big)& \le |a +\lambda b | \textbf{1}_{\{ 0< |a + \lambda b | <\xi \}}  \notag \\
 & \le |a +\lambda b | \in L^1(\Omega \times(0,T)\times \R^d),
\end{align*}
for $m(dz)$-almost every $z\in \R$. Again, $|a +\lambda b | \bf{1}_{\{ 0< |a + \lambda b | <\xi \}}\longrightarrow 0 $ as $\xi \goto 0$
for almost every $(t,x)$ and almost surely. Since $\int_{|z|>0} (1\wedge |z|^2) m(dz) < \infty $, by dominated convergence theorem, we conclude that 
$\mathcal{G} \goto 0$ as $\xi \goto 0$.

Combining all the above estimates, we arrive at 
\begin{align}
 \E\Big[ \int_{\R^d} \beta_\xi \big(u_\eps(t,x)-v_\eps(t,x)\big)\,dx\Big]& \le
 - \frac{ \eps}{2} \E\Big[ \int_{\R^d}
 \int_0^t \beta_\xi^{\prime\prime} \big( u_\eps(s,x)-v_\eps(s,x)\big) \Big|\grad \big(u_\eps(s,x)-v_\eps(s,x)\big)\Big|^2  \,ds\,dx\Big] \notag \\
  & \qquad + \eps(\xi) + \E\Big[ \int_{\R^d} \beta_\xi \big(u_0(x)-v_0(x)\big)\,dx\Big]\notag \\
  & \le  \eps(\xi) + \E\Big[ \int_{\R^d} \beta_\xi \big(u_0(x)-v_0(x)\big)\,dx\Big].
  \label{esti:final-bv}
\end{align}
 Keeping $\eps >0$ fixed, we pass to the limit $\xi\goto 0$ in \eqref{esti:final-bv} and the resulting expressions reads as
\begin{align*}
 \E\Big[ \int_{\R^d} \big|u_\eps(t,x)-v_\eps(t,x)\big|\,dx\Big] \le \E\Big[ \int_{\R^d} \big|u_0(x)-v_0(x)\big|\,dx\Big].
\end{align*} 
Assume that $v_0(x)=u_0(x+c)$ for fixed $c\in \R^d$. Then, since $\sigma$ and $\eta$ do not depend on $x$ explicitly, by uniqueness of the weak solution, one can conclude that $v_\eps(t,x)=u_\eps(t,x+c)$ and hence
 \begin{align*}
 \E\Big[ \int_{\R^d}  \frac{ \big|u_\eps(t,x)-v_\eps(t,x)\big|}{|c|}\,dx\Big]& \le \E\Big[ \int_{\R^d} \frac{\big|u_0(x)-u_0(x+c)\big|}{|c|}\,dx\Big]\le C, 
 \end{align*}
independent of $c$, if $u_0 \in BV(\R^d)$. This implies that, for any $t>0$, since $v_\eps(t,x)=u_\eps(t,x+c)$
 \begin{align}
\sup_{\eps>0} \E\Big[TV_x(u_\eps(t))] \le  \E\Big[TV_x(u_0)\Big] . \label{bv-bound}
 \end{align}
 This completes the proof.
 \end{proof}
 
 In view of the well-posedness results from \cite{BaVaWit_2014,BisMajVal}, one can conclude that, under the assumptions \ref{A1}-\ref{A6}, the family $\{u_\eps(t,x)\}_{\eps>0}$ converges to the Young measure valued narrow limit process $u(t,x,\varsigma),$ called generalized entropy solution which is indeed the unique entropy solution $u(t,x)$ of the underlying problem \eqref{eq:stoc_con_brown}. Now, our aim is to show that $u(t,x)$ is actually a spatial $BV$ solution of \eqref{eq:stoc_con_brown} provided the initial function $u_0$ lies in $L^2 \cap BV(\R^d)$. 
Since $u_\eps$ converges to $u$ weakly in $L^2(\Omega\times(0,T)\times\R^d)$, for any $R>0$, by convexity arguments, 
\begin{align*}
\E\Big[\int_{\Pi_T}|u| {\bf 1}_{B_{\R^d}(0,R)}\,dx\,dt\Big] \leq \lim\inf_\eps \E\Big[\int_{\Pi_T}|u_\eps| {\bf 1}_{B_{\R^d}(0,R)}\,dx\,dt\Big] \leq M,
\end{align*}
thanks to \eqref{L1-bound} and $u \in L^1(\Omega\times(0,T)\times\R^d)$.
In view of the lower semi-continuity property of $TV_x$ and 
Fatou's lemma, we have, for a.e. $t>0$,
\begin{align*}
 \E\Big[ TV_x(u(t))\Big] \le \liminf_{\eps \goto 0} \E\Big[ TV_x(u_\eps(t))\Big] \le \E \Big[ TV_x(u_0)\Big],
\end{align*}
where the last inequality follows from Theorem \eqref{thm:bv-viscous}. Thus, $u(t,x)$ is a function of bounded variation in spatial variable.
In other words, we have existence of BV entropy solution for the problem \eqref{eq:stoc_con_brown} given by the following theorem.
 \begin{thm} [BV entropy solution]\label{thm:existence-bv}
 Suppose that the assumptions \ref{A1}-\ref{A6} hold. Then there exists a constant $C>0$, and an unique entropy solution
 of \eqref{eq:stoc_con_brown} such that for a.e. $t>0$
\begin{align*}
   \E \Big[|u(t,\cdot)|_{BV(\R^d)} \Big] \le C\E \Big[|u_0|_{BV(\R^d)} \Big].
\end{align*}	
\end{thm} 

\begin{rem}\label{estimGradAu}
Note that $\nabla A(u_\eps)$ converges to $\nabla A(u)$ weakly in $L^2(\Omega\times(0,T)\times\R^d)$. Moreover,
\begin{align*}
\E\Big[ \int_{\Pi_T} |\nabla A(u_\eps)|\Big] = \E\Big[ \int_{\Pi_T} A^\prime(u_\eps)|\nabla u_\eps|\Big] \leq \|A^\prime\|_\infty \E\Big[ \int_{\Pi_T} |\nabla u_\eps|\Big] \leq C
\end{align*}
thanks to the estimate of the total variation of $u_\eps$. Then, an argument similar to the above one concerning the sequence $u_\eps$ yields $\nabla A(u) \in L^1(\Omega\times(0,T)\times\R^d)$.
\end{rem}

 
 \section{Proof of the Main Theorem}
 \label{sec:main-thm}
It is worth mentioning that, the average $L^1$-contraction principle (cf. \cite{BaVaWit_2014, BisMajVal}) gives the continuous dependence on the initial data in stochastic balance laws of the type \eqref{eq:stoc_con_brown}. However, we intend to
establish continuous dependence also on the nonlinearities, i.e., on the flux functions and the noise coefficients.
To achieve that, we proceed as follows: 
For $\eps>0$, let $v_\eps$ be a weak solution to the problem 
  \begin{align}
  dv_\eps(s,y) -\Delta  B(v_\eps(s,y))\,ds - \mbox{div}_y g(v_\eps(s,y)) \,ds&=\widetilde{\sigma}(v_\eps(s,y))dW(s) + \int_{|z|>0} \widetilde{\eta}(v_\eps(s,y);z)\,\widetilde{N}(dz,ds) \notag \\
   & \hspace{4cm}+ \eps \Delta_{yy} v_{\eps}(s,y),\label{eq:viscous} \\
  v_\eps(0,y)&=v_0(y)\notag
  \end{align}
  and $u_\theta(t,x)$ be a weak solution to the viscous problem \eqref{eq:viscous-Brown} with small positive parameter $\theta$ which is
  different from $\eps$.
  In view of the Theorem \ref{thm:existence-bv}, we see that $v_\eps(s,y)$ converges to the unique BV entropy solution $v(s,y)$ 
  of \eqref{eq:stoc_con_brown-1} with initial data $v_0(y)$  and $u_\theta(t,x)$ converges to $u(t,x)$ which is 
   the unique BV entropy solution to the problem \eqref{eq:stoc_con_brown}. Our aim is to derive  expected value of
   the $L^1$-norm of $u-v$ and the proof is done by adapting the method of ``{\it doubling of variables}" to the stochastic case.
In \cite{BisKoleyMaj,Chen-karlsen_2012}, the authors directly compare one entropy solution $u(t,x)$ to
the viscous solutions $v_\eps(s,y)$ and then pass to the limit in a Kato's inequality. Due to lack of regularity
of the solution (see e.g. estimation of the term $\mathcal{A}_1$), here we compare one weak solution $u_\theta(t,x)$ to another weak solution $v_\eps(s,y)$ and then pass to the limits 
as viscous parameters tend to zero. This approach is somewhat different from the deterministic approach, where one can directly compare two entropy solutions. For deterministic continuous dependence theory consult \cite{perthame,cockburn,chenkarlsen_2005,kal-resibro} and references therein.

Note that, one can show that $v_\eps \in H^1(\R^d)$. However, to prove such Kato inequality (see \cite{BaVaWit_2014,BisMajKarl_2014}), one typically requires higher regularity of $v_\eps$. Therefore, we need to regularize $v_\eps$ by convolution.
Let $ \{\tau_\kappa\} $ be a sequence of mollifier in $\R^d$. Since $v_\eps$ is a viscous solution to the problem \eqref{eq:viscous}, one gets that $v_{\eps} \con \tau_\kappa$ is a solution to the problem
\begin{align}
   d (v_\eps \con \tau_\kappa) -\Delta (B(v_\eps)\con \tau_\kappa )\,ds =& \mbox{div}_y (g(v_\eps)\con \tau_\kappa) \,ds
   + (\widetilde{\sigma}( v_\eps)\con \tau_\kappa )\,dW(s) + \int_{|z|> 0} (\widetilde{\eta}( v_\eps;z)\con \tau_\kappa )\widetilde{N}(dz,ds)  \notag \\
   & + \eps \Delta( v_\eps \con \tau_\kappa(s,y))\,ds, \quad s>0, ~ y\in \R^d. \label{eq:viscous-regularize}\end{align}
Note that, $ \Delta (v_\eps \con \tau_\kappa) \in L^2(\Omega \times \Pi_T)$, for fixed $\eps>0$.

To proceed further, let $\rho$ and $\varrho$ be the standard mollifiers on $\R$ and  $\R^d$ respectively such that $\supp(\rho) \subset [-1,0)$ and $\supp(\varrho) = B_1(0)$. For $\delta > 0$ and $\delta_0 > 0$, let $\rho_{\delta_0}(r) = \frac{1}{\delta_0}\rho(\frac{r}{\delta_0})$ and
$\varrho_{\delta}(x) = \frac{1}{\delta^d}\varrho(\frac{x}{\delta})$. Let  $\psi\in C_c^{1,2}([0,\infty)\times \rd)$ be any nonnegative test function. For two positive constants $\delta, \delta_0 $, we define the test function        
  \begin{align}
  \label{eq:doubled-variable} \varphi_{\delta,\delta_0}(t,x, s,y) = \rho_{\delta_0}(t-s) \varrho_{\delta}(x-y) \psi(s,y). 
  \end{align}  
 Furthermore, let $\varsigma$ be the standard symmetric nonnegative mollifier on $\R$ with support in $[-1,1]$ and $\varsigma_l(r)= \frac{1}{l} \varsigma(\frac{r}{l})$ for $l > 0$.
We write down the It\^{o}-L\'{e}vy formula for weak solution $u_\theta(t,x)$ against the convex entropy flux triple $(\beta_\xi(\cdot-k), f^{\beta_\xi}(\cdot, k),A^{\beta_\xi}(\cdot,k))$, multiply by $\varsigma_l(v_\eps\con \tau_\kappa(s,y)-k)$ for $k\in \R$, and then integrate with respect to $ s, y$ and $k$. The result is, keeping in mind that $\beta=\beta_{\xi}$
 \begin{align}
 &  \theta \,\E\Big[\int_{\Pi_T^2} \int_{\R} \beta^{\prime \prime }(u_{\theta}(t,x)-k) |\grad u_\theta(t,x))|^2
 \varphi_{\delta,\delta_0}(t,x,s,y)\varsigma_l(v_\eps\con \tau_\kappa(s,y)-k)\,dk\,dx\,dt\,dy\,ds\Big]  \notag \\
 &  +  \E\Big[\int_{\Pi_T^2} \int_{\R} \beta^{\prime \prime }\big( u_\theta(t,x)-k\big) |\grad G(u_\theta(t,x))|^2
 \varphi_{\delta,\delta_0}(t,x,s,y)\varsigma_l(v_\eps\con \tau_\kappa(s,y)-k)\,dk\,dx\,dt\,dy\,ds\Big] \notag \\
  & \le  \E\Big[\int_{\Pi_T}\int_{\R^d}\int_{\R} \beta(u_0(x)-k)\varphi_{\delta,\delta_0}(0,x,s,y) \varsigma_l(v_{\eps}\con \tau_\kappa(s,y)-k)\,dk \,dx\,dy\,ds\Big] \notag \\
 & + \E\Big[ \int_{\Pi_T^2} \int_{\R} \beta(u_\theta(t,x)-k)\partial_t \varphi_{\delta,\delta_0}(t,x,s,y)
 \varsigma_l(v_{\eps}\con \tau_\kappa(s,y)-k)\,dk \,dx\,dt\,dy\,ds\Big]\notag \\ 
  & + \E\Big[\int_{\Pi_T^2}\int_{\R}  \sigma(u_\theta(t,x)) \beta^\prime (u_\theta(t,x)-k) \varphi_{\delta,\delta_0}(t,x,s,y)\varsigma_l(v_{\eps}\con \tau_\kappa(s,y)-k)\,dk \,dx\,dW(t) \,dy\,ds\Big]\notag\\
 & + \frac{1}{2}\E\Big[\int_{\Pi_T^2} \int_{\R} \sigma^2(u_\theta(t,x))
 \beta^{\prime\prime}(u_\theta(t,x)-k) \varphi_{\delta,\delta_0}(t,x,s,y) \varsigma_l(v_{\eps}\con \tau_\kappa(s,y)-k)dk\,dx\,dt\,dy\,ds\Big]\notag \\
 & - \E\Big[\int_{\Pi_T^2} \int_{\R}  f^\beta(u_\theta(t,x),k)\grad_x \varphi_{\delta,\delta_0}(t,x,s,y)
 \varsigma_l(v_{\eps}\con \tau_\kappa(s,y)-k)\,dk\,dx\,dt\,dy\,ds\Big] \notag \\
 & + \E\Big[\int_{\Pi_T^2} \int_{\R}  A^\beta(u_\theta(t,x),k)\Delta_x \varphi_{\delta,\delta_0}(t,x,s,y) \varsigma_l(v_{\eps}\con \tau_\kappa(s,y)-k)\,dk\,dx\,dt\,dy\,ds\Big] \notag \\
 &  - \theta \,\E\Big[\int_{\Pi_T^2} \int_{\R} \beta^{\prime}(u_{\theta}(t,x)-k) \grad_x u_{\theta}(t,x)\cdot \grad_x \varphi_{\delta,\delta_0}(t,x,s,y)
 \varsigma_l(v_\eps\con \tau_\kappa(s,y)-k)\,dk\,dx\,dt\,dy\,ds\Big] \notag \\
& + \E\Big[\int_{\Pi_T^2}\int_{|z|>0}\int_{\R} \int_0^1  \eta(u_\theta(t,x);z) \beta^\prime \big(u_\theta(t,x)+ \lambda  \eta(u_\theta(t,x);z)
  -k\big) \varphi_{\delta,\delta_0}(t,x,s,y) \notag \\
  & \hspace{7cm} \times \varsigma_l(v_{\eps}\con \tau_\kappa(s,y)-k)\,d\lambda\,dk \,\widetilde{N}(dz,dt)\,dx \,dy\,ds\Big]\notag\\
 &+ \E\Big[\int_{\Pi_T^2} \int_{\R} \int_{|z|>0} \int_0^1 (1-\lambda)\eta^2(u_\theta(t,x);z)
 \beta^{\prime\prime}\big(u_\theta(t,x)+\lambda \eta(u_\theta(t,x);z)-k\big) \varphi_{\delta,\delta_0}(t,x,s,y) \notag \\
 &  \hspace{7cm} \times \varsigma_l(v_{\eps}\con \tau_\kappa(s,y)-k)\,d\lambda\,m(dz)\,dx\,dt\,dy\,ds\Big]\notag \\
 & \text{ i.e.,} \qquad  I_{0,1} + I_{0,2} \le  I_1 + I_2 + I_3 +I_4 + I_5 + I_6 + I_7 +I_8 + I_9. \label{stoc_entropy_1}
 \end{align}
 Again we apply It\^{o} formula to \eqref{eq:viscous-regularize} and multiply with the test function $\varphi_{\delta,\delta_0}$
 and $\varsigma_l(u_{\theta}(t,x)-k)$. Taking expectation and integrating with respect to $k,t$ and $x$, the resulting inequality reads, for $\beta=\beta_{\xi}$ 
 
 \begin{align}
 &\eps \,\E\Big[\int_{\Pi_T^2} \int_{\R} \beta^{\prime \prime}(v_\eps\con \tau_\kappa(s,y) -k) |\grad(v_\eps\con \tau_\kappa)|^2
 \varphi_{\delta,\delta_0} \varsigma_l(u_\theta(t,x)-k)\,dk\,dx\,dt\,dy\,ds\Big] \notag \\
 +&  \E\Big[\int_{\Pi_T^2} \int_{\R} \beta^{\prime \prime}(v_\eps\con \tau_\kappa(s,y) -k)
 \grad( B(v_\eps)\con \tau_\kappa)\cdot \grad(v_\eps\con \tau_\kappa)
 \varphi_{\delta,\delta_0}(t,x,s,y) \varsigma_l(u_\theta(t,x)-k)\,dk\,dx\,dt\,dy\,ds\Big] \notag \\
 \le &  \E \Big[\int_{\Pi_T}\int_{\R^d}\int_{\R} \beta(v_\eps\con \tau_\kappa(0,y)-k)\varphi_{\delta,\delta_0}
 (t,x,0,y) \varsigma_l(u_\theta(t,x)-k)\,dk\,dx\,dy\,dt\Big] \notag \\
 & + \E \Big[ \int_{\Pi_T^2}\int_{\R} \beta(v_\eps\con \tau_\kappa(s,y)-k)\partial_s \varphi_{\delta,\delta_0}
 \varsigma_l(u_\theta(t,x)-k)\,dk \,dy\,ds\,dx\,dt\Big]\notag \\ 
 & + \E \Big[\int_{\Pi_T^2} \int_{\R}  (\widetilde{\sigma}(v_\eps)\con \tau_\kappa(s,y)) \beta^\prime(v_\eps\con \tau_\kappa(s,y)-k)
  \varphi_{\delta,\delta_0}(t,x,s,y)\varsigma_l(u_\theta(t,x)-k)\,dy\,dk \,dW(s)\,dx\,dt\Big]\notag\\
 & + \frac{1}{2} \E \Big[\int_{\Pi_T^2}  \int_{\R}  (\widetilde{\sigma}(v_\eps)\con \tau_\kappa(s,y))^2
 \beta^{\prime\prime}(v_\eps\con \tau_\kappa(s,y)-k) \varphi_{\delta,\delta_0}(t,x,s,y)\varsigma_l(u_\theta(t,x)-k)\,dy\,dk\,ds\,dx\,dt\Big]\notag \\
  &- \E \Big[\int_{\Pi_T^2} \int_{\R}\beta^\prime(v_\eps\con \tau_\kappa(s,y)-k) \grad(B(v_\eps)\con \tau_\kappa) 
 \grad_y \varphi_{\delta,\delta_0}(t,x,s,y)  \varsigma_l(u_\theta(t,x)-k)\,dk\,dx\,dt\,dy\,ds\Big] \notag \\
 & -\E \Big[\int_{\Pi_T^2}\int_{\R} \beta^\prime(v_\eps\con \tau_\kappa(s,y)-k)(g(v_\eps)\con \tau_\kappa(s,y)) \grad_y
 \varphi_{\delta,\delta_0}(t,x,s,y) \varsigma_l(u_\theta(t,x)-k)\,dk\,dx\,dt\,dy\,ds\Big] \notag \\
 &- \E \Big[\int_{\Pi_T^2} \int_{\R} \beta^{\prime\prime}(v_\eps\con \tau_\kappa(s,y)-k) (g(v_\eps)\con \tau_\kappa) \grad_y
 (v_\eps \con \tau_\kappa) \varphi_{\delta,\delta_0}(t,x,s,y) \varsigma_l(u_\theta(t,x)-k)\,dk\,dx\,dt\,dy\,ds\Big] \notag \\
 &- \eps\, \E \Big[\int_{\Pi_T^2} \int_{\R} \beta^\prime(v_\eps\con \tau_\kappa(s,y)-k)
 \grad_y (v_\eps\con \tau_\kappa)\cdot\grad_y  \varphi_{\delta,\delta_0}(t,x,s,y)\varsigma_l(u_\theta(t,x)-k)\,dk \,dy\,ds\,dx\,dt\Big] \notag \\
 & + \E \Big[\int_{\Pi_T^2} \int_{|z|>0} \int_{\R}\int_0^1  (\widetilde{\eta}(v_\eps;z)\con \tau_\kappa(s,y)) \,\beta_\xi^\prime\big(v_\eps\con \tau_\kappa(s,y) + \lambda
  (\widetilde{\eta}(v_\eps;z)\con \tau_\kappa(s,y))-k\big) \notag \\
  & \hspace{6cm} \times \varphi_{\delta,\delta_0}(t,x,s,y)\varsigma_l(u_\theta(t,x)-k)\,d\lambda\,dk\,\widetilde{N}(dz,ds)\,dy\,dx\,dt\Big]\notag\\
 & + \E \Big[\int_{\Pi_T^2}  \int_{\R}\int_{|z|>0}\int_0^1 (1-\lambda)  (\widetilde{\eta}(v_\eps;z)\con \tau_\kappa(s,y))^2
 \beta_\xi^{\prime\prime}\big(v_\eps\con \tau_\kappa(s,y) + \lambda (\widetilde{\eta}(v_\eps;z)\con \tau_\kappa(s,y)) -k\big) \notag \\
 &\hspace{6cm} \times \varphi_{\delta,\delta_0}(t,x,s,y)\varsigma_l(u_\theta(t,x)-k) \,d\lambda\,m(dz)\,dk\,dy\,ds\,dx\,dt\Big] \notag \\
 & \text{i.e.,} \qquad J_{0,1} + J_{0,2} \le  J_1 + J_2 + J_3 + J_4 + J_5 + J_6 + J_7 + J_8 + J_9 + J_{10}. \label{stoc_entropy_2}
 \end{align} 
Our aim is to add inequalities \eqref{stoc_entropy_1} and \eqref{stoc_entropy_2}, and pass to the limits with respect to the various parameters involved. We do this by claiming a series
of lemma's and proofs of these lemmas follow from \cite{BaVaWit_2014,BisMajKarl_2014} modulo cosmetic changes.

\begin{lem}\label{lem:additional degenerate-terms}
Let $\widetilde{G}(x)=\int_0^x \sqrt{B^{\prime}(r)}\,dr$. Then the following holds
\begin{align*}
 	\lim_{l\goto 0}\lim_{\kappa \goto 0}\lim_{\delta_0\goto 0} \big(I_{0,2} + J_{0,2}\big)
 	&= \E\Big[\int_{\Pi_T}\int_{\R^d} \beta_{\xi}^{\prime\prime} \big( u_\theta(t,x)-v_{\eps}(t,y)\big) 
 	\Big( |\grad G(u_\theta(t,x))|^2 + |\grad \widetilde{G}(v_\eps(t,y))|^2 \Big) \notag \\
 	& \hspace{6cm} \times \psi(t,y) \varrho_{\delta}(x-y) \,dx\,dy\,dt \Big].
\end{align*}
\end{lem}
 
\begin{lem}\label{lem:initial+time-terms}
It holds that $J_1=0$ and 
\begin{align*}
 \lim_{l\goto 0}\lim_{\kappa \goto 0}\lim_{\delta_0\goto 0} \big(I_1 + J_1\big)& =  \E \Big[\int_{\R^d \times \R^d}
 	\beta_{\xi}(u_0(x)-v_0(y))\psi(0,y)\varrho_{\delta} (x-y)\,dx\,dy\Big]\notag.\\
 	\lim_{l\goto 0}\lim_{\kappa \goto 0}\lim_{\delta_0\goto 0} \big(I_2 + J_2\big) 
 	&=  \E \Big[\int_{\Pi_T}\int_{\R^d} \beta_\xi \big(v_\eps(s,y)-u_\theta(s,x)\big) \partial_s\psi(s,y)
 	\, \varrho_\delta(x-y)\,dy\,dx\,ds\Big].
\end{align*}
\end{lem}
 
\begin{lem}\label{lem:stochastic-terms}
We have $J_3=0 =J_9$ and the following hold:
\begin{align*}
 	 \lim_{l\goto 0}\lim_{\kappa \goto 0}\lim_{\delta_0\goto 0}	\Big( \big(I_3 + J_3\big) + \big(J_4 + I_4 \big)\Big)
 	 &= \frac{1}{2} \E\Big[\int_{\Pi_T}\int_{\R^d} \beta_\xi^{\prime\prime}
 	 \big(u_\theta(t,x)-v_{\eps}(t,y)\big) \big(\sigma(u_\theta(t,x))-\widetilde{\sigma}(v_\eps(t,y))\big)^2 \notag \\
 	 &  \hspace{5cm} \times \psi(t,y)\varrho_{\delta}(x-y)\,dx\,dy\,dt \Big],
 	\end{align*}
 	and 
  	\begin{align*}
 	 &\lim_{l\goto 0}\lim_{\kappa \goto 0}\lim_{\delta_0\goto 0}\Big(I_8 + I_9 + J_9 + J_{10}\Big) \notag \\
 	 &= \E\Big[\int_{\Pi_T}\int_{\R^d} \int_{|z|>0} \int_0^1 (1-\lambda) \beta_\xi^{\prime\prime}
 	 \Big(u_\theta(t,x)-v_{\eps}(t,y) + \lambda \big( \eta(u_\theta(t,x);z)- \widetilde{\eta}(v_\eps(t,y);z)\big)\Big) \notag \\
 	&  \hspace{4cm} \times \big( \eta(u_\theta(t,x);z)-\widetilde{\eta}(v_\eps(t,y);z)\big)^2
 	  \psi(t,y)\varrho_{\delta}(x-y) \,d\lambda\,m(dz)\,dx\,dy\,dt \Big].
\end{align*}
\end{lem}
 
\begin{lem}\label{lem:flux-terms}
The following hold
\begin{align*}
 	\lim_{l\goto 0}\lim_{\kappa \goto 0}\lim_{\delta_0\goto 0} \big(J_6 + J_7\big) &= - \E\Big[\int_{\Pi_T}\int_{\R^d} g^{\beta_\xi} \big(v_\eps(t,y),u_\theta(t,x)\big) \cdot \grad_y [\psi(t,y)\varrho_{\delta}(x-y)]\,dx\,dy\,dt \Big]\\
 	\lim_{l\goto 0}\lim_{\kappa \goto 0}\lim_{\delta_0\goto 0} I_5 &= - \E\Big[\int_{\Pi_T}\int_{\R^d} f^{\beta_\xi} \big(u_\theta(t,x),v_\eps(t,y)\big) \cdot \grad_x [\psi(t,y)\varrho_{\delta}(x-y)]\,dx\,dy\,dt \Big]\\
 \lim_{l\goto 0}\lim_{\kappa \goto 0}\lim_{\delta_0\goto 0} \big(I_7 + J_8\big) &
 \le \frac{C}{\delta} \Big\{  \eps\, \E\Big[|v_0|_{BV(\R^d)}\Big] +  \theta\, \E\Big[|u_0|_{BV(\R^d)}\Big]\Big\}.	
\end{align*}
\end{lem}
 
\begin{lem}\label{lem:degenerate-terms}
Let $\beta=\beta_{\xi}$ as prescribed in Section \ref{sec:tech}. Let $B^\beta(a,b)=\int_b^a \beta^\prime(r-b)B^\prime(r)\,dr$. Then, we have the following:
 \begin{align*}
\lim_{l\goto 0}\lim_{\kappa \goto 0}\lim_{\delta_0\goto 0} I_6 &= \E\Big[\int_{\Pi_T}\int_{\R^d} A^\beta \big(u_\theta(t,x),v_\eps(t,y)\big)\Delta_x [\psi(t,y)\varrho_{\delta}(x-y)]\,dx\,dy\,dt \Big]\\
\lim_{l\goto 0}\lim_{\kappa \goto 0}\lim_{\delta_0\goto 0} J_5 &= \E\Big[\int_{\Pi_T}\int_{\R^d} B^\beta \big(v_\eps(t,y),u_\theta(t,x)\big)\Delta_y [\psi(t,y)\varrho_{\delta}(x-y)]\,dx\,dy\,dt \Big].
\end{align*}
\end{lem}
Finally, note that $I_{0,1}, J_{0,1}$ are non-negative quantities. Now we are in a position to
add the inequalities \eqref{stoc_entropy_1} and \eqref{stoc_entropy_2}, and pass to the
limits $\lim_{l\goto 0}\lim_{\kappa \goto 0}\lim_{\delta_0\goto 0}$. Thanks to the Lemmas
\ref{lem:additional degenerate-terms}-\ref{lem:flux-terms} and Lemma \ref{lem:degenerate-terms}, we arrive at 
 \begin{align}
 0 \le& - 
 \E\Big[\int_{\Pi_T}\int_{\R^d} \beta_\xi^{\prime\prime}\big( u_\theta(t,x)-v_{\eps}(t,y)\big) \Big( |\grad G(u_\theta(t,x))|^2 + |\grad \widetilde{G}(v_\eps(t,y))|^2 \Big) 
  \psi(t,y)\varrho_{\delta}(x-y) \,dx\,dy\,dt \Big] \notag
   \\& +  
   \E\Big[\int_{\Pi_T}\int_{\R^d} A^{\beta_\xi} \big(u_\theta(t,x),v_\eps(t,y)\big)\Delta_x [\psi(t,y)\varrho_{\delta}(x-y)]\,dx\,dy\,dt \Big] \notag 
   \\& +  
   \E\Big[\int_{\Pi_T}\int_{\R^d} B^{\beta_\xi} \big(v_\eps(t,y),u_\theta(t,x)\big)\Delta_y [\psi(t,y)\varrho_{\delta}(x-y)]\,dx\,dy\,dt \Big] \notag 
   \\& + 
   \E\Big[\int_{\Pi_T}\int_{\R^d} \grad_y\cdot \big\{ g^{\beta_\xi} \big(v_\eps(t,y),u_\theta(t,x)\big)- f^{\beta_\xi} \big(u_\theta(t,x),v_\eps(t,y)\big) \big\}  \psi(t,y)\varrho_{\delta}(x-y)\,dx\,dy\,dt \Big] \notag 
   \\&  - 
   \E\Big[\int_{\Pi_T}\int_{\R^d} f^{\beta_\xi} \big(u_\theta(t,x),v_\eps(t,y)\big) \cdot \nabla_{y} \psi(t,y)\varrho_{\delta}(x-y)\,dx\,dy\,dt \Big] \notag 
   \\& +  \frac{1}{2} 
   \E\Big[\int_{\Pi_T}\int_{\R^d} \beta_\xi^{\prime\prime} \big(u_\theta(t,x)-v_{\eps}(t,y)\big) \big(\sigma(u_\theta(t,x))-\widetilde{\sigma}(v_\eps(t,y))\big)^2 \psi(t,y)\varrho_{\delta}(x-y)\,dx\,dy\,dt \Big]\notag 
   \\& +  
   \E \Big[\int_{\Pi_T}\int_{\R^d} \beta_\xi(v_\eps(s,y)-u_\theta(s,x)) \partial_s\psi(s,y)\varrho_\delta(x-y)\,dy\,dx\,ds\Big]\notag 
   \\& +  
   \E \Big[\int_{\R^d \times \R^d} \beta_\xi(u_0(x)-v_0(y))\psi(0,y)\varrho_{\delta} (x-y)\,dx\,dy\Big] \notag \\
   &+  \E\Big[\int_{\Pi_T}\int_{\R^d} \int_{|z|>0} \int_0^1 (1-\lambda) \beta_\xi^{\prime\prime}
 	 \Big(u_\theta(t,x)-v_{\eps}(t,y) + \lambda \big( \eta(u_\theta(t,x);z)- \widetilde{\eta}(v_\eps(t,y);z)\big)\Big) \notag \\
 	&  \hspace{4cm} \times \big( \eta(u_\theta(t,x);z)-\widetilde{\eta}(v_\eps(t,y);z)\big)^2
 	  \psi(t,y)\varrho_{\delta}(x-y) \,d\lambda\,m(dz)\,dx\,dy\,dt \Big] \notag \\
   & \hspace{5cm} + \frac{C}{\delta} \Big\{ \eps \, \E \Big[|v_0|_{BV(\R^d)}\Big] + \theta \, \E \Big[|u_0|_{BV(\R^d)}\Big]\Big\} \notag 
  \\& := \mathcal{A}_1 + \mathcal{A}_2 + \mathcal{A}_3 + \mathcal{A}_4 + \mathcal{A}_5 + \mathcal{A}_6
  + \mathcal{A}_7 + \mathcal{A}_8 + \mathcal{A}_9 + \frac{C}{\delta} \Big\{ \eps \, \E \Big[|v_0|_{BV(\R^d)}\Big] + \theta \, \E \Big[|u_0|_{BV(\R^d)}\Big]\Big\}. \label{stoc_entropy_3}
 \end{align}

Let us first consider $\mathcal{A}_1$. Note that for any $a,b \in \R, -(a^2 + b^2) \le - 2 ab$. Thus, in view of 
Lipschitz property of $G$ and $\widetilde{G}$ and the fact that $u_\theta(t,\cdot), v_\eps(t,\cdot) \in H^1(\R^d)$, we see that
\begin{align*}
\mathcal{A}_1=& 
- \E\Big[\int_{\Pi_T}\int_{\R^d} \beta_\xi^{\prime\prime}\big( u_\theta(t,x)-v_{\eps}(t,y)\big) \Big( |\grad G(u_\theta(t,x))|^2 + |\grad \widetilde{G}(v_\eps(t,y))|^2 \Big) 
  \psi(t,y)\varrho_{\delta}(x-y) \,dx\,dy\,dt \Big]
\\ = &
- \E\Big[\int_{\Pi_T}\int_{\R^d} \beta_\xi^{\prime\prime}\big( u_\theta(t,x)-v_{\eps}(t,y)\big) \Big( A'(u_\theta(t,x))|\grad u_\theta(t,x)|^2
+ B'(v_\eps(t,y))|\grad v_\eps(t,y)|^2 \Big) \\
  & \hspace{6cm} \times \psi(t,y)\varrho_{\delta}(x-y) \,dx\,dy\,dt \Big]
\\ \leq &
-2 \, \E\Big[\int_{\Pi_T}\int_{\R^d} \beta_\xi^{\prime\prime}\big( u_\theta(t,x)-v_{\eps}(t,y)\big) \Big( \sqrt{A'(u_\theta(t,x))}\sqrt{B'(v_\eps(t,y))}
\grad_x u_\theta(t,x)\grad_y v_\eps(t,y) \Big) \\
  & \hspace{8cm} \times  \psi(t,y)\varrho_{\delta}(x-y) \,dx\,dy\,dt \Big].
\end{align*}
Regarding the term $\mathcal{A}_2$, we have
\begin{align*}
\mathcal{A}_2=&  
   \E \Big[\int_{\Pi_T}\int_{\R^d} A^{\beta_\xi} \big(u_\theta(t,x),v_\eps(t,y)\big)\Delta_x [\psi(t,y)\varrho_{\delta}(x-y)]\,dx\,dy\,dt \Big]
   \\ =&
-   \E\Big[\int_{\Pi_T}\int_{\R^d} \beta^\prime_\xi(u_\theta(t,x)-v_\eps(t,y))A'(u_\theta(t,x))\nabla_x u_\theta(t,x)\nabla_x [\psi(t,y)\varrho_{\delta}(x-y)]\,dx\,dy\,dt \Big]
   \\ =&
\E\Big[\int_{\Pi_T}\int_{\R^d} \beta^\prime_\xi(u_\theta(t,x)-v_\eps(t,y))A'(u_\theta(t,x))\nabla_x u_\theta(t,x)\nabla_y [\psi(t,y)\varrho_{\delta}(x-y)]\,dx\,dy\,dt \Big]
   \\ &
-\E\Big[\int_{\Pi_T}\int_{\R^d} \beta^\prime_\xi(u_\theta(t,x)-v_\eps(t,y))A'(u_\theta(t,x))\nabla_x u_\theta(t,x)\varrho_{\delta}(x-y) \nabla_y\psi(t,y)\,dx\,dy\,dt \Big]
   \\ =&
\E\Big[\int_{\Pi_T}\int_{\R^d} \beta^{\prime\prime}_\xi(u_\theta(t,x)-v_\eps(t,y))A'(u_\theta(t,x))\nabla_x u_\theta(t,x)\nabla_y v_\eps(t,y) [\psi(t,y)\varrho_{\delta}(x-y)]\,dx\,dy\,dt \Big]
   \\ &
-\E\Big[\int_{\Pi_T}\int_{\R^d} \beta^\prime_\xi(u_\theta(t,x)-v_\eps(t,y))A'(u_\theta(t,x))\nabla_x u_\theta(t,x)\varrho_{\delta}(x-y) \nabla_y\psi(t,y)\,dx\,dy\,dt \Big]
\\:=& \mathcal{A}_{2,1} + \mathcal{A}_{2,2}.
\end{align*}
Similarly, for the term $\mathcal{A}_3$, we have
\begin{align*}
\mathcal{A}_3=&  
   \E\Big[\int_{\Pi_T}\int_{\R^d} B^{\beta_\xi} \big(v_\eps(t,y),u_\theta(t,x)\big)\Delta_y [\psi(t,y)\varrho_{\delta}(x-y)]\,dx\,dy\,dt \Big]
\\ = &
-\E\Big[\int_{\Pi_T}\int_{\R^d} \beta^\prime_\xi(v_\eps(t,y)-u_\theta(t,x))B'(v_\eps(t,y))\nabla_y v_\eps(t,y)\nabla_y [\psi(t,y)\varrho_{\delta}(x-y)]\,dx\,dy\,dt \Big]   
\\ = &
\E\Big[\int_{\Pi_T}\int_{\R^d} \beta^\prime_\xi(v_\eps(t,y)-u_\theta(t,x))B'(v_\eps(t,y))\nabla_y v_\eps(t,y)\nabla_x [\psi(t,y)\varrho_{\delta}(x-y)]\,dx\,dy\,dt \Big]   
\\&
-\E\Big[\int_{\Pi_T}\int_{\R^d} \beta^\prime_\xi(v_\eps(t,y)-u_\theta(t,x))B'(v_\eps(t,y))\nabla_y v_\eps(t,y)\nabla_y \psi(t,y) \varrho_{\delta}(x-y) \,dx\,dy\,dt \Big]   
\\ = &
\E\Big[\int_{\Pi_T}\int_{\R^d} \beta^{\prime\prime}_\xi(v_\eps(t,y)-u_\theta(t,x))B'(v_\eps(t,y))\nabla_y v_\eps(t,y)\nabla_x u_\theta(t,x) [\psi(t,y)\varrho_{\delta}(x-y)]\,dx\,dy\,dt \Big]   
\\&
-\E\Big[\int_{\Pi_T}\int_{\R^d} \beta^\prime_\xi(v_\eps(t,y)-u_\theta(t,x))B'(v_\eps(t,y))\nabla_y v_\eps(t,y)\nabla_y \psi(t,y) \varrho_{\delta}(x-y) \,dx\,dy\,dt \Big]   
\\:=& \mathcal{A}_{3,1} + \mathcal{A}_{3,2}.
\end{align*}
Thus,
\begin{align*}
&\mathcal{A}_{1} +  \mathcal{A}_{2,1} + \mathcal{A}_{3,1} 
\\ \leq &
-2 \,\E\Big[\int_{\Pi_T}\int_{\R^d} \beta_\xi^{\prime\prime}\big( u_\theta(t,x)-v_{\eps}(t,y)\big) \Big( \sqrt{A'(u_\theta(t,x))}\sqrt{B'(v_\eps(t,y))}\grad_x u_\theta(t,x)\grad_y
v_\eps(t,y) \Big) \\
   & \hspace{9cm} \times \psi(t,y)\varrho_{\delta}(x-y) \,dx\,dy\,dt \Big]
\\ & +
\E\Big[\int_{\Pi_T}\int_{\R^d} \beta^{\prime\prime}_\xi(u_\theta(t,x)-v_\eps(t,y))A'(u_\theta(t,x))\nabla_x u_\theta(t,x)\nabla_y v_\eps(t,y) [\psi(t,y)\varrho_{\delta}(x-y)]\,dx\,dy\,dt \Big]
\\&+
\E\Big[\int_{\Pi_T}\int_{\R^d} \beta^{\prime\prime}_\xi(v_\eps(t,y)-u_\theta(t,x))B'(v_\eps(t,y))\nabla_y v_\eps(t,y)\nabla_x u_\theta(t,x) [\psi(t,y)\varrho_{\delta}(x-y)]\,dx\,dy\,dt \Big] 
 \\ =&
\E\Big[\int_{\Pi_T}\int_{\R^d} \beta_\xi^{\prime\prime}\big( u_\theta(t,x)-v_{\eps}(t,y)\big)  \big[\sqrt{A'(u_\theta(t,x))}- \sqrt{B'(v_\eps(t,y))}\big]^2\grad_x u_\theta(t,x)\grad_y v_\eps(t,y)  \\
  & \hspace{9cm} \times \psi(t,y)\varrho_{\delta}(x-y) \,dx\,dy\,dt \Big], 
\end{align*}
if $\beta_\xi$ is chosen even. This implies that 
\begin{align*}
&\mathcal{A}_{1} +  \mathcal{A}_{2,1} + \mathcal{A}_{3,1} 
\\ \leq &
\E\Big[\int_{\Pi_T}\int_{\R^d} \nabla_x\Bigg\{\int_{v_\eps(t,y)}^{u_\theta(t,x)}\beta_\xi^{\prime\prime}\big( \tau-v_{\eps}(t,y)\big)  \big[\sqrt{A'(\tau)}- \sqrt{B'(v_\eps(t,y))}\big]^2\,d\tau\Bigg\} \\
& \hspace{8cm} \times \grad_y v_\eps(t,y) \psi(t,y)\varrho_{\delta}(x-y) \,dx\,dy\,dt \Big] 
\\ = &  
-\E\Big[\int_{\Pi_T}\int_{\R^d} \Bigg\{\int_{v_\eps(t,y)}^{u_\theta(t,x)}\beta_\xi^{\prime\prime}\big( \tau-v_{\eps}(t,y)\big)  \big[\sqrt{A'(\tau)}- \sqrt{B'(v_\eps(t,y))}\big]^2\,d\tau\Bigg\} \\
& \hspace{8cm} \times
\grad_y v_\eps(t,y)  \nabla_x\varrho_{\delta}(x-y)
  \psi(t,y) \,dx\,dy\,dt \Big] 
\\ \leq &  \Bigg|
2\, \E\Big[\int_{\Pi_T}\int_{\R^d} \Bigg\{\int_{v_\eps(t,y)}^{u_\theta(t,x)}\beta_\xi^{\prime\prime}\big( \tau-v_{\eps}(t,y)\big)  \big[\sqrt{A'(\tau)}- \sqrt{B'(\tau)}\big]^2\,d\tau\Bigg\} \\
& \hspace{8cm} \times
\grad_y v_\eps(t,y)  \nabla_x\varrho_{\delta}(x-y)
  \psi(t,y) \,dx\,dy\,dt \Big] 
  \\+ &
  2\, \E\Big[\int_{\Pi_T}\int_{\R^d} \Bigg\{\int_{v_\eps(t,y)}^{u_\theta(t,x)}\beta_\xi^{\prime\prime}\big( \tau-v_{\eps}(t,y)\big)  \big[\sqrt{B'(\tau)}- \sqrt{B'(v_\eps(t,y))}\big]^2\,d\tau\Bigg\} \\
 & \hspace{8cm} \times \grad_y v_\eps(t,y)
  \nabla_x\varrho_{\delta}(x-y) \psi(t,y) \,dx\,dy\,dt \Big] \Bigg| 
\\ \leq &  
C(\|A'-B'\|_\infty+\xi) \E\Big[\int_{\Pi_T}\int_{\R^d}  |\beta_\xi^{\prime}\big( u_\theta(t,x)-v_{\eps}(t,y)| |\grad_y v_\eps(t,y)|  |\nabla_x\varrho_{\delta}(x-y)|
  \psi(t,y) \,dx\,dy\,dt \Big], 
\end{align*}
where we have used that $|\sqrt{x}-\sqrt{y}| \leq \sqrt{2|x-y|}$ and the Lipschitz continuity of $B'$. Thus, 
\begin{align*}
\mathcal{A}_{1} +  \mathcal{A}_{2,1} + \mathcal{A}_{3,1} 
 \leq &
C(\|A'-B'\|_\infty+\xi) \E \Big[\int_{\Pi_T}  \|\psi(t, \cdot)\|_\infty\, |\grad_y v_\eps(t,y)| \int_{\R^d} |\nabla_x\varrho_{\delta}(x-y)| \,dx\,dy\,dt \Big]
\\
\leq & C\frac{(\|A'-B'\|_\infty+\xi)}{\delta} \int_{0}^T ||\psi(t,\cdot)||_{L^\infty(\R^d)}\,dt,
\end{align*}
thanks to the uniform BV estimate of $v_\eps$.

Now, since
\begin{align*}
\mathcal{A}_{2,2}+\mathcal{A}_{3,2}=& -\E\Big[\int_{\Pi_T}\int_{\R^d} \beta^\prime_\xi(u_\theta(t,x)-v_\eps(t,y))A'(u_\theta(t,x))\nabla_x u_\theta(t,x)\varrho_{\delta}(x-y) \nabla_y\psi(t,y)\,dx\,dy\,dt \Big]
\\&
-\E\Big[\int_{\Pi_T}\int_{\R^d} \beta^\prime_\xi(v_\eps(t,y)-u_\theta(t,x))B'(v_\eps(t,y))\nabla_y v_\eps(t,y)\nabla_y \psi(t,y) \varrho_{\delta}(x-y) \,dx\,dy\,dt \Big], 
\end{align*}
we have proved the following lemma.
\begin{lem}\label{estim_A1_3}
\begin{align}\label{esti:a1+a2+a3-final}
&\mathcal{A}_{1}+\mathcal{A}_{2}+\mathcal{A}_{3}  \leq 
C\frac{(\|A'-B'\|_\infty+\xi)}{\delta}  \int_0^T \|\psi(t,\cdot)\|_{L^\infty(\R^d)} \,dt \notag
\\& -\E\Big[\int_{\Pi_T}\int_{\R^d} \beta^\prime_\xi(u_\theta(t,x)-v_\eps(t,y))[\nabla_x A(u_\theta(t,x)) -\nabla_y B(v_\eps(t,y))]  \varrho_{\delta}(x-y) \nabla_y\psi(t,y)\,dx\,dy\,dt \Big]. 
\end{align}
\end{lem}

Let us consider the term $\mathcal{A}_4$. We first rewrite $\mathcal{A}_4$ as 
\begin{align*}
 \mathcal{A}_4 =&\E\Big[\int_{\Pi_T}\int_{\R^d} \grad_y\cdot \bigg\{ g^{\beta_\xi} \big(v_\eps(t,y),u_\theta(t,x)\big)- f^{\beta_\xi} \big(u_\theta(t,x),v_\eps(t,y)\big) \bigg\}  \psi(t,y)\varrho_{\delta}(x-y)\,dx\,dy\,dt \Big] 
 \\=&
 - \E \Big[\int_{\Pi_T}\int_{\R_y^d} \nabla_y v_\eps(s,y) \cdot \partial_v \big(
 f^\beta(u,v)-g^\beta(v,u)\big)\Big|_{(u,v)=(u_\theta(s,x),v_\eps(s,y))}
 \psi(s,y)\varrho_\delta(x-y)\,dy\,dx\,ds\Big].
 \end{align*}
 Therefore, to estimate $\mathcal{A}_4$, it is required to estimate $ \partial_v \big(f^\beta(u,v)-g^\beta(v,u)\big)\Big|_{(u,v)=(u_\theta(s,x),v_\eps(s,y))}$. Note that, by our choice of $\beta=\beta_\xi$, one has
  \begin{align}
  \Big| \frac{\partial}{\partial v}\Big(f^{\beta_\xi}(u,v)-f^{\beta_\xi}(v,u)\Big)\Big|
  &=\Big|- f^{\prime}(v)\beta_\xi^{\prime}(v-u) - f^{\prime}(v)\beta_\xi^{\prime}(0) + 
  \int_{s=u}^v \beta_\xi^{\prime\prime}(s-v) f^\prime(s)\,ds\Big| \notag  \\
  &=\Big| \big(f^{\prime}(v)-f^{\prime}(u) \big)\beta_\xi^{\prime}(u-v) -
  \int_{s=u}^v \beta_\xi^{\prime}(s-v) f^{\prime\prime}(s)\,ds\Big| \notag \\
  & = \Big| \int_{u}^v \Big(\beta_\xi^{\prime}(u-v) -\beta_\xi^{\prime}(s-v) \Big) f^{\prime\prime}(s)\,ds \Big|
  \le M_2\,\xi\, ||f^{\prime\prime}||_{\infty}.\label{esti:deri-entropy-flux}
\end{align} 
Again, it is evident that, for any $u\in \R$
  \begin{align}
  \Big| \frac{\partial}{\partial v}\Big(f^\beta(v,u)-g^\beta(v,u)\Big)\Big| =  
   \Big| \beta_{\xi}^{\prime}(v-u) \big( f^\prime(v)-g^\prime(v)\big)\Big|
  \le |f^\prime(v)-g^\prime(v)|.  \label{esti:deri-different-entropy-flux}
  \end{align}
  Therefore, by \eqref{esti:deri-entropy-flux} and \eqref{esti:deri-different-entropy-flux}, we obtain
  \begin{align}
  \Big| \frac{\partial}{\partial v}\Big(f^\beta(u,v)-g^\beta(v,u)\Big)\Big|\le  M_2\,\xi\, ||f^{\prime\prime}||_{\infty}
  + |f^\prime(v)-g^\prime(v)| \label{esti:deri-different-entropy-flux-1}
  \end{align}
  Using uniform spatial BV bound and the estimate \eqref{esti:deri-different-entropy-flux-1}, we obtain 
\begin{lem}\label{estim_A4}
  \begin{align}
\mathcal{A}_4 \le  \E\Big[|v_0|_{BV(\R^d)}\Big] \Big( M_2\,\xi\, ||f^{\prime\prime}||_{\infty}
  + ||f^\prime-g^\prime||_{\infty}\Big) \int_{s=0}^T ||\psi(s,\cdot)||_{L^\infty(\R^d)}\,ds.\label{esti:a4-final}
  \end{align}
\end{lem}
\noindent Next, regarding the term $\mathcal{A}_5$, we have
\begin{lem}\label{estim_A5}
\begin{align}
\mathcal{A}_5  &= - 
   \E\Big[\int_{\Pi_T}\int_{\R^d} f^{\beta_\xi} \big(u_\theta(t,x),v_\eps(t,y)\big) \cdot \grad_y \psi(t,y)\varrho_{\delta}(x-y)\,dx\,dy\,dt \Big] 
.\label{esti:a5-final}
\end{align} 
\end{lem}

\noindent Regarding $\mathcal{A}_6$, in view of the definition of the 
\begin{align*}
 \mathcal{E}(\sigma, \widetilde{\sigma}):= \sup_{\xi \neq0} \frac{|\sigma(\xi)-\widetilde{\sigma}(\xi)|}{|\xi|},
\end{align*}
assumption \ref{A4}, and Remark~\ref{beta}, we see that 
\begin{align*}
\mathcal{A}_6  =& \frac{1}{2} 
   \E\Big[\int_{\Pi_T}\int_{\R^d} \beta_\xi^{\prime\prime} \big(u_\theta(t,x)-v_{\eps}(t,y)\big) \big(\sigma(u_\theta(t,x))-\widetilde{\sigma}(v_\eps(t,y))\big)^2 \psi(t,y)\varrho_{\delta}(x-y)\,dx\,dy\,dt \Big]\notag 
\\\le& 
\E \Big[\int_{\Pi_T}\int_{\R^d} \Big(\sigma(u_\theta(t,x))-\widetilde{\sigma}(u_\theta(t,x))\Big)^2 \beta_\xi^{\prime\prime} \big(u_\theta(t,x)-v_\eps(t,y)\big) \psi(t,y)\varrho_{\delta}(x-y)\,dx\,dy\,dt \Big] \notag \\
& + \E \Big[\int_{\Pi_T}\int_{\R^d} \Big(\widetilde{\sigma}(u_\theta(t,x))-\widetilde{\sigma}(v_\eps(t,y))\Big)^2 \beta_\xi^{\prime\prime} \big(u_\theta(t,x)-v_\eps(t,y)\big) \psi(t,y)\varrho_{\delta}(x-y)\,dx\,dy\,dt \Big]
\\ \le& 
\big( \mathcal{E}(\sigma, \widetilde{\sigma}) \big)^2\, \E \Big[\int_{\Pi_T}\int_{\R^d} |u_\theta(t,x)|^2 \beta_\xi^{\prime\prime} \big(u_\theta(t,x)-v_\eps(t,y)\big) \psi(t,y)\varrho_{\delta}(x-y)\,dx\,dy\,dt \Big] \notag \\
& +2 \E \Big[\int_{\Pi_T}\int_{\R^d} \beta_\xi \big(u_\theta(t,x)-v_\eps(t,y)\big) \psi(t,y)\varrho_{\delta}(x-y)\,dx\,dy\,dt \Big],
\end{align*}
and, in view of the above inequality and $\beta_{\xi}^{\prime\prime}(r)\le \frac{C}{\xi}$, we obtain
\begin{lem}\label{estim_A6}
\begin{align}
\mathcal{A}_6 & \le C \frac{\big( \mathcal{E}(\sigma, \widetilde{\sigma}) \big)^2}{\xi}  \int_{s=0}^T ||\psi(s,\cdot)||_{L^\infty(\R^d)}\,ds 
+2 \E \Big[\int_{\Pi_T}\int_{\R^d} \beta_\xi \big(u_\theta(t,x)-v_\eps(t,y)\big) \psi(t,y)\varrho_{\delta}(x-y)\,dx\,dy\,dt \Big].
\label{esti:a6-final}
\end{align}
\end{lem}

\noindent  Next we consider the term $\mathcal{A}_8$. Since $\beta_\xi(r)\le |r|$, we obtain 
\begin{align}
\mathcal{A}_8 &= \E \Big[\int_{\R^d \times \R^d} \beta_\xi(u_0(x)-v_0(y))\psi(0,y)\varrho_{\delta} (x-y)\,dx\,dy\Big]\notag
\\&\le  
\E \Big[\int_{\R^d \times \R^d} \big|u_0(x)-v_0(y)\big|\psi(0,y)\varrho_{\delta} (x-y)\,dx\,dy\Big].\label{esti:a8-final}
\end{align}

\noindent Let us focus on the term $\mathcal{A}_9$. For this, let us define
\begin{align*}
 a:= u_\theta(t,x)-v_\eps(t,y),\,\, \text{and}\,\,\, b:= \eta(u_\theta(t,x);z)-\widetilde{\eta}(v_\eps(t,y);z).
\end{align*}
We can now rewrite $\mathcal{A}_9$ in the following simplified form 
\begin{align}
 \mathcal{A}_9 & = \E\Big[\int_{\Pi_T}\int_{\R^d} \int_{|z|>0} \int_0^1 (1-\lambda) b^2 \beta_\xi^{\prime\prime}
 \big(a + \lambda b\big) \psi(t,y)\varrho_{\delta}(x-y) \,d\lambda\,m(dz)\,dx\,dy\,dt \Big] \notag \\
  & \le C  \E\Big[\int_{\Pi_T}\int_{\R^d} \int_{|z|>0} \int_0^1 (1-\lambda) \big| \eta(u_\theta(t,x);z)-\widetilde{\eta}(u_\theta(t,x);z)\big|^2
 \beta_\xi^{\prime\prime}\big(a + \lambda b\big) \psi(t,y) \notag \\
 & \hspace{6cm} \times \varrho_{\delta}(x-y) \,d\lambda\,m(dz)\,dx\,dy\,dt \Big] \notag \\
 &  + C \E\Big[\int_{\Pi_T}\int_{\R^d} \int_{|z|>0} \int_0^1 (1-\lambda) \big| \widetilde{\eta}(u_\theta(t,x);z)-\widetilde{\eta}(v_\eps(t,y);z)\big|^2
 \beta_\xi^{\prime\prime}\big(a + \lambda b\big) \psi(t,y) \notag \\
 & \hspace{7cm} \times \varrho_{\delta}(x-y) \,d\lambda\,m(dz)\,dx\,dy\,dt \Big] \notag \\
 & := \mathcal{A}_{9,1} + \mathcal{A}_{9,2}.\label{esti:a90}
\end{align}
Note that $\beta_{\xi}^{\prime\prime}(r)\le \frac{C}{\xi}$, for any $r\in \R$. 
Thus, in view of the definition of
\begin{align*}
\mathcal{D}(\eta,\widetilde{\eta})= \sup_{u\neq 0} \int_{|z|>0} \frac{\big|\eta(u,z)-\widetilde{\eta}(u,z)\big|^2}{|u|^2} \,m(dz),
\end{align*}
and the uniform moment estimate \eqref{bounds:a-priori-viscous-solution}, we see that 
\begin{align}
 \mathcal{A}_{9,1}  &  \le \frac{C}{\xi} \mathcal{D}(\eta,\widetilde{\eta}) \E\Big[\int_{\Pi_T}\int_{\R^d}  |u_\theta(t,x)|^2 
 \psi(t,y)\varrho_{\delta}(x-y)\,dx\,dy\,dt \Big]  \notag \\
 & \le  \frac{C}{\xi} \mathcal{D}(\eta,\widetilde{\eta}) \int_{0}^T ||\psi(s,\cdot)||_{L^{\infty}(\R^d)}\,ds.\label{esti:a91}
\end{align}
Next we move on to estimate the term $\mathcal{A}_{9,2}$. Notice that, thanks to the assumption \ref{A5}, 
\begin{align}
  \big| \widetilde{\eta}(u_\theta(t,x);z)-\widetilde{\eta}(v_\eps(t,y);z)\big|^2 \beta_\xi^{\prime\prime}\big(a + \lambda b\big) & \le 
   \big|u_\theta(t,x)-v_\eps(t,y)\big|^2  \beta_\xi^{\prime\prime}\big(a + \lambda b\big) \big(1 \wedge |z|^2\big) \notag  \\
   & \le a^2 \beta_\xi^{\prime\prime}\big(a + \lambda b\big) \big(1 \wedge |z|^2\big).\label{esti:a92-1}
\end{align}
Therefore, we need to find a suitable upper bound on $a^2 \beta_\xi^{\prime\prime}\big(a + \lambda b\big)$. Here we follow the similar argument 
as we have done in Section \ref{sec:apriori+existence} (estimation of the term $\mathcal{G}$). Since $\beta^{\prime\prime}$
is nonnegative and symmetric around zero, we can assume without loss of generality that $a \ge 0$. Thus, by assumption \ref{A5}, we have
\begin{align*}
 |b| & \le \big|\eta(u_\theta(t,x);z)-\widetilde{\eta}(u_\theta(t,x);z)\big| + \big|\widetilde{\eta}(u_\theta(t,x);z)-\widetilde{\eta}(v_\eps(t,y);z)\big| \notag \\
 &  \le \big|\eta(u_\theta(t,x);z)-\widetilde{\eta}(u_\theta(t,x);z)\big| + \lambda^* \big|u_\theta(t,x)-v_\eps(t,y)\big| \notag \\
 & = \big|\eta(u_\theta(t,x);z)-\widetilde{\eta}(u_\theta(t,x);z)\big| + \lambda^* a,
\end{align*}
and hence 
\begin{align*}
 -\lambda^* a - \big|\eta(u_\theta(t,x);z)-\widetilde{\eta}(u_\theta(t,x);z)\big|  \le -|b|. 
\end{align*}
Thus, for $\lambda \in [0,1]$, we see that 
  $(1-\lambda^*) a - \big|\eta(u_\theta(t,x);z)-\widetilde{\eta}(u_\theta(t,x);z)\big|  \le a+ \lambda b$, which gives 
  \begin{align}
   0\le a \le (1-\lambda^*)^{-1} \Big\{ (a+ \lambda b) +  \big|\eta(u_\theta(t,x);z)-\widetilde{\eta}(u_\theta(t,x);z)\big|  \Big\}. \label{esti:a92-2}
  \end{align}
Making use of \eqref{esti:a92-2} in \eqref{esti:a92-1} along with Remark \ref{beta} and the fact that  $\beta_{\xi}^{\prime\prime}(r)\le \frac{C}{\xi}$,  we get 
\begin{align*}
a^2 \beta_\xi^{\prime\prime}\big(a + \lambda b\big) \big(1 \wedge |z|^2\big) 
\leq &
\frac{2\big(1 \wedge |z|^2\big)}{(1-\lambda^*)^{2}}\Big\{ (a+ \lambda b)^2 +  \big|\eta(u_\theta(t,x);z)-\widetilde{\eta}(u_\theta(t,x);z)\big|^2  \Big\} \beta_\xi^{\prime\prime}\big(a + \lambda b\big) 
\\ \leq &
\frac{2\big(1 \wedge |z|^2\big)}{(1-\lambda^*)^{2}}\Big\{ 2\beta_\xi\big(a + \lambda b\big)  +  \frac{C}{\xi}\big|\eta(u_\theta(t,x);z)-\widetilde{\eta}(u_\theta(t,x);z)\big|^2  \Big\}. 
\end{align*}
Let us remark now that 
\begin{align*}
\beta_\xi\big(a + \lambda b\big) =& \beta_\xi\big(|a + \lambda b|\big) \leq \beta_\xi\big(2|a| + |\eta(u_\theta(t,x);z)-\widetilde{\eta}(u_\theta(t,x);z)|\big)
\\ \leq &
2\beta_\xi\big(a\big) + |\eta(u_\theta(t,x);z)-\widetilde{\eta}(u_\theta(t,x);z)|,
\end{align*}
to get:
 \begin{align*}
  \mathcal{A}_{9,2} 
   \le &
  C \E\Big[\int_{\Pi_T}\int_{\R^d} \int_{|z|>0}  
   \Bigg\{ \beta_\xi\big(u_\theta(t,x)-v_\eps(t,y)\big) + \big|\eta(u_\theta(t,x);z)-\widetilde{\eta}(u_\theta(t,x);z)\big| \notag \\
  & \hspace{2,5cm} + \frac{\big|\eta(u_\theta(t,x);z)-\widetilde{\eta}(u_\theta(t,x);z)\big|^2}{\xi}\Bigg\}  
    \times \big(1 \wedge |z|^2\big)  \psi(t,y) \varrho_{\delta}(x-y)\,m(dz)\,dx\,dy\,dt \Big] \notag \\
\\ \le&
  C \E\Big[\int_{\Pi_T}\int_{\R^d} 
 \beta_\xi\big(u_\theta(t,x)-v_\eps(t,y)\big)
   \psi(t,y) \varrho_{\delta}(x-y)\,dx\,dy\,dt \Big] \notag 
   \\
  &+ C \E\Big[\int_{\Pi_T}\int_{\R^d} 
   \Bigg\{  \int_{|z|>0}  \frac{\big|\eta(u_\theta(t,x);z)-\widetilde{\eta}(u_\theta(t,x);z)\big|}{|u_\theta(t,x)|}
   \big(1 \wedge |z|\big) \,m(dz)\Bigg\}  \notag \\
   & \hspace{8,5cm} \times |u_\theta(t,x)|  \psi(t,y) \varrho_{\delta}(x-y)\,dx\,dy\,dt \Big] \notag \\
   \\
  &+ \frac{C}{\xi} \E\Big[\int_{\Pi_T}\int_{\R^d} \mathcal{D}(\eta,\widetilde{\eta}) |u_\theta(t,x)|^2
  \psi(t,y) \varrho_{\delta}(x-y)\,dx\,dy\,dt \Big] \notag
\\  \le&
  C \E\Big[\int_{\Pi_T}\int_{\R^d} 
 \beta_\xi\big(u_\theta(t,x)-v_\eps(t,y)\big)
   \psi(t,y) \varrho_{\delta}(x-y)\,dx\,dy\,dt \Big] \notag 
   \\
  &+ C\sqrt{\mathcal{D}(\eta, \widetilde{\eta})} \,\E\Big[\int_{\Pi_T}\int_{\R^d}    |u_\theta(t,x)|  \psi(t,y) \varrho_{\delta}(x-y)\,dx\,dy\,dt \Big] 
  + \frac{C}{\xi} \mathcal{D}(\eta,\widetilde{\eta})  \int_0^T
\|\psi(t,\cdot)\|_{L^\infty(\R^d)}\,dt.
   \end{align*}
 Thus, 
 \begin{align}
\mathcal{A}_{9,2} 
  \le& 
C \E\Big[\int_{\Pi_T}\int_{\R^d} 
 \beta_\xi\big(u_\theta(t,x)-v_\eps(t,y)\big)
   \psi(t,y) \varrho_{\delta}(x-y)\,dx\,dy\,dt \Big] \notag
   \\&
   + C\Big(\sqrt{\mathcal{D}(\eta, \widetilde{\eta})} 
  + \frac{\mathcal{D}(\eta,\widetilde{\eta}) }{\xi}\Big)  \int_0^T
\|\psi(t,\cdot)\|_{L^\infty(\R^d)}\,dt.
\label{esti:a92}
 \end{align}
 Thus, combining \eqref{esti:a91} and \eqref{esti:a92} in \eqref{esti:a90}, we obtain the following bound:
 \begin{align}
  \mathcal{A}_9 
  \le& 
C \E\Big[\int_{\Pi_T}\int_{\R^d} 
 \beta_\xi\big(u_\theta(t,x)-v_\eps(t,y)\big)
   \psi(t,y) \varrho_{\delta}(x-y)\,dx\,dy\,dt \Big] \notag
   \\&
   + C\Big(\sqrt{\mathcal{D}(\eta, \widetilde{\eta})} 
  + \frac{\mathcal{D}(\eta,\widetilde{\eta}) }{\xi}\Big)  \int_0^T
\|\psi(t,\cdot)\|_{L^\infty(\R^d)}\,dt.
\label{esti:a9-final}
 \end{align}
Finally, invoking the estimates \eqref{esti:a1+a2+a3-final}, \eqref{esti:a4-final}, \eqref{esti:a5-final}, \eqref{esti:a6-final}, \eqref{esti:a8-final} and 
\eqref{esti:a9-final} in \eqref{stoc_entropy_3} we have 
\begin{align}
 0 \le &  \E \Big[\int_{\R^d} \int_{\R^d} \big|u_0(x)-v_0(y)\big|\psi(0,y)\varrho_{\delta} (x-y)\,dx\,dy\Big] \label{esti:final-01}  
 \\& + 
 \E \Big[\int_{\Pi_T}\int_{\R^d} \beta_\xi \big(u_\theta(t,x)-v_\eps(t,y)\big) \partial_t\psi(t,y)\varrho_\delta(x-y)\,dy\,dx\,dt\Big]\notag 
 \\& + 
 C\E \Big[\int_{\Pi_T}\int_{\R^d} \beta_\xi \big(u_\theta(t,x)-v_\eps(t,y)\big) \psi(t,y)\varrho_\delta(x-y)\,dy\,dx\,dt\Big]\notag 
 \\&
 + C \Bigg(\frac{\mathcal{E}(\sigma, \widetilde{\sigma})^2}{\xi} + \xi + ||f^\prime-g^\prime||_{\infty} + \sqrt{\mathcal{D}(\eta, \widetilde{\eta})}+ \frac{\mathcal{D}(\eta, \widetilde{\eta})}{\xi} \Bigg) \int_{0}^T ||\psi(t,\cdot)||_{L^\infty(\R^d)}\,dt  \notag \\
 &+  C\frac{(\|A'-B'\|_\infty+\xi)}{\delta} \int_{0}^T ||\psi(t,\cdot)||_{L^\infty(\R^d)}\,dt 
 + \frac{C}{\delta} \big( \eps +\theta\big) \notag \\
&- \E \Big[\int_{\Pi_T}\int_{\R^d} f^{\beta_\xi} \big(u_\theta(t,x),v_\eps(t,y)\big) \cdot \grad_y \psi(t,y)\varrho_{\delta}(x-y)\,dx\,dy\,dt \Big]   \notag
\\ &-\E\Big[\int_{\Pi_T}\int_{\R^d} \beta^\prime_\xi(u_\theta(t,x)-v_\eps(t,y))[\nabla_x A(u_\theta(t,x)) -  \nabla_y B(v_\eps(t,y))]  \varrho_{\delta}(x-y) \nabla_y\psi(t,y)\,dx\,dy\,dt \Big].\notag
\end{align}
Note that $\{u_\theta\}_{\theta>0}$ converges  in $L^p_{loc}(\R^d, L^p((0,T)\times\Omega))$ for any $p\in[1,2)$ to the unique BV entropy solution $u$, $\{v_\eps\}_{\eps>0}$ converges in the same way to the unique BV entropy solution $v$ of \eqref{eq:stoc_con_brown-1} with initial data $v_0$ and $A(u_\theta)$ and $B(v_\eps)$ converge weakly in $L^2(\Omega\times(0,T),H^1(\R^d))$. Thus, by
passing to the limit as $\eps, \theta \goto 0$ in \eqref{esti:final-01}, we obtain 
\begin{align}
 0 & \le   \E \Big[\int_{\R^d} \int_{\R^d} \big|u_0(x)-v_0(y)\big|\psi(0,y)\varrho_{\delta} (x-y)\,dx\,dy\Big]  \label{stoc_entropy_4}  
 \\& + 
 \E \Big[\int_{\Pi_T}\int_{\R^d} \beta_\xi \big(u(t,x)-v(t,y)\big) \partial_t\psi(t,y)\varrho_\delta(x-y)\,dy\,dx\,dt\Big]\notag 
 \\& + 
 C\E \Big[\int_{\Pi_T}\int_{\R^d} \beta_\xi \big(u(t,x)-v(t,y)\big) \psi(t,y)\varrho_\delta(x-y)\,dy\,dx\,dt\Big]\notag 
 \\&
 + C \Bigg(\frac{\mathcal{E}(\sigma, \widetilde{\sigma})^2}{\xi} + \xi + ||f^\prime-g^\prime||_{\infty} + \sqrt{\mathcal{D}(\eta, \widetilde{\eta})}+ \frac{\mathcal{D}(\eta, \widetilde{\eta})}{\xi} + \frac{(\|A'-B'\|_\infty+\xi)}{\delta} \Bigg) \int_{0}^T ||\psi(t,\cdot)||_{L^\infty(\R^d)}\,dt  \notag
\\&
-  \E\Big[\int_{\Pi_T}\int_{\R^d} f^{\beta_\xi} \big(u(t,x),v(t,y)\big) \cdot \grad_y \psi(t,y)\varrho_{\delta}(x-y)\,dx\,dy\,dt \Big]   \notag
\\ &-\E\Big[\int_{\Pi_T}\int_{\R^d} \beta^\prime_\xi(u(t,x)-v(t,y))[\nabla_x A(u(t,x)) - \nabla_y B(v(t,y))]  \varrho_{\delta}(x-y) \nabla_y\psi(t,y)\,dx\,dy\,dt \Big]
.\notag
\end{align}
To proceed further, we make a special choice for the function $\psi(t,x)$. To this end, for each $h>0$ and fixed $t\ge 0$, we define
 \begin{align}
 \psi_h^t(s)=\begin{cases} 1, &\quad \text{if}~ s\le t, \notag \\
 1-\frac{s-t}{h}, &\quad \text{if}~~t\le s\le t+h,\notag \\
 0, & \quad \text{if} ~ s \ge t+h.
 \end{cases}
 \end{align}
Furthermore, let $\zeta \in C_c^2(\R^d)$ be any nonnegative test function. 
%
\noindent Clearly, \eqref{stoc_entropy_4} holds with
 $\psi(s,x)=\psi_h^t(s)\zeta(x)$. Let $\mathbb{T}$ be the set all points $t$ in $[0, \infty)$ such that $t$ is a right
 Lebesgue point of 
 \begin{align*}
\mathcal{B}(t)= \E \Big[\int_{\R^d}\int_{\R^d} \beta_\xi\big(u(t,x)-v(t,y)\big)\zeta(y)
\varrho_\delta(x-y)\,dx\,dy\Big].
\end{align*}

\noindent Clearly, $\mathbb{T}^{\complement}$ has zero Lebesgue measure. Fix  $t\in \mathbb{T}$. Thus, we have, from \eqref{stoc_entropy_4}
\begin{align*}
& \frac{1}{h}\int_{t}^{t+h}  \E \Big[\int_{\R^{2d}} \beta_\xi\big( u(s,x)-v(s,y)\big)\zeta(y) \varrho_\delta(x-y)\,dx\,dy \Big]\,ds \notag
\\ &\le   \E \Big[\int_{\R^{2d}}  \big|u_0(x)-v_0(y)\big|\zeta(y)\varrho_{\delta} (x-y)\,dx\,dy\Big]  
 \\&
+ C \E\Big[\int_0^{t+h}\int_{\R^{2d}} \beta_\xi\big( u(s,x)-v(s,y)\big) \zeta(y) \psi_h^t(s) \varrho_{\delta}(x-y)\,dx\,dy\,ds \Big]   \notag
\\&
 + C \Bigg(\frac{\mathcal{E}(\sigma, \widetilde{\sigma})^2}{\xi} + \xi + ||f^\prime-g^\prime||_{\infty} 
 + \sqrt{\mathcal{D}(\eta, \widetilde{\eta})}+ \frac{\mathcal{D}(\eta, \widetilde{\eta})}{\xi} 
 + \frac{(\|A'-B'\|_\infty+\xi)}{\delta}\Bigg) ||\zeta(\cdot)||_{L^\infty} \int_{s=0}^{t+h} \psi_h^t(s)  \,ds \notag
\\&
-  \E\Big[\int_0^{t+h}\int_{\R^{2d}} f^{\beta_\xi} \big(u(s,x),v(s,y)\big) \cdot \grad \zeta(y) \psi_h^t(s) \varrho_{\delta}(x-y)\,dx\,dy\,ds \Big]   \notag
\\ &-\E\Big[\int_0^{t+h}\int_{\R^{2d}} \beta^\prime_\xi(u(s,x)-v(s,y))[\nabla_x A(u(s,x)) - \nabla_y B(v(s,y))]  \varrho_{\delta}(x-y) \nabla\zeta(y)\psi_h^t(s)\,dx\,dy\,ds \Big].\notag
\end{align*}
Passing to the limit as $h\goto 0$, we obtain
 \begin{align*}
&\E \Big[\int_{\R^{2d}} \beta_\xi\big( u(t,x)-v(t,y)\big)\zeta(y) \varrho_\delta(x-y)\,dx\,dy \Big] \notag
\\ \le &  \E \Big[\int_{\R^{2d}}  \big|u_0(x)-v_0(y)\big|\zeta(y)\varrho_{\delta} (x-y)\,dx\,dy\Big]  
\\&
+ C \E\Big[\int_0^{t}\int_{\R^{2d}}  \beta_\xi\big( u(t,x)-v(t,y)\big) \zeta(y)  \varrho_{\delta}(x-y)\,dx\,dy\,ds \Big]   \notag
 \\&
 + C \Bigg(\frac{\mathcal{E}(\sigma, \widetilde{\sigma})^2}{\xi} + \xi + ||f^\prime-g^\prime||_{\infty} + \sqrt{\mathcal{D}(\eta, \widetilde{\eta})} + \frac{\mathcal{D}(\eta, \widetilde{\eta})}{\xi} + \frac{(\|A'-B'\|_\infty+\xi)}{\delta} \Bigg) ||\zeta(\cdot)||_{L^\infty(\R^d)} t \notag
\\&
-  \E\Big[\int_0^{t}\int_{\R^{2d}} f^{\beta_\xi} \big(u(s,x),v(s,y)\big) \cdot \grad \zeta(y)  \varrho_{\delta}(x-y)\,dx\,dy\,ds \Big]   \notag
\\ &-\E\Big[\int_0^{t}\int_{\R^{2d}} \beta^\prime_\xi(u(s,x)-v(s,y))[\nabla_x A(u(s,x)) - \nabla_y B(v(s,y))]  \varrho_{\delta}(x-y) \nabla\zeta(y)\,dx\,dy\,ds \Big].\notag
\end{align*}
By then sending  $\zeta$  to ${\bf 1}_{\R^d}$ (thanks to Remark \ref{estimGradAu} for the last term), we have
\begin{align*}
\E \Big[\int_{\R^{2d}} & \beta_\xi\big( u(t,x)-v(t,y)\big) \varrho_\delta(x-y)\,dx\,dy \Big] \notag
\\ \le &  \E \Big[\int_{\R^{2d}}  \big|u_0(x)-v_0(y)\big|\varrho_{\delta} (x-y)\,dx\,dy\Big]  
 \\&
 + C \Bigg(\frac{\mathcal{E}(\sigma, \widetilde{\sigma})^2}{\xi} + \xi + ||f^\prime-g^\prime||_{\infty} + \sqrt{\mathcal{D}(\eta, \widetilde{\eta})}+ \frac{\mathcal{D}(\eta, \widetilde{\eta})}{\xi} 
+  \frac{\|A'-B'\|_\infty+\xi}{\delta} \Bigg) t
\\&
+ C \E\Big[\int_0^{t}\int_{\R^{2d}}  \beta_\xi\big( u(t,x)-v(t,y)\big)  \varrho_{\delta}(x-y)\,dx\,dy\,ds \Big],   \notag 
\end{align*}
and
\begin{align*}
\E \Big[\int_{\R^{2d}} & \beta_\xi\big( u(t,x)-v(t,y)\big) \varrho_\delta(x-y)\,dx\,dy \Big] \notag
\\ \le &  e^{Ct}\E \Big[\int_{\R^{2d}}  \big|u_0(x)-v_0(y)\big|\varrho_{\delta} (x-y)\,dx\,dy\Big]  
 \\&
 + Ce^{Ct} \Bigg(\frac{\mathcal{E}(\sigma, \widetilde{\sigma})^2}{\xi} + \xi + ||f^\prime-g^\prime||_{\infty} + \sqrt{\mathcal{D}(\eta, \widetilde{\eta})}+ \frac{\mathcal{D}(\eta, \widetilde{\eta})}{\xi} 
+  \frac{\|A'-B'\|_\infty+\xi}{\delta} \Bigg) t
\end{align*}
by Gronwall argument.
Let us consider now a bounded by $1$ weight-function $\Phi\in L^1(\R^d)$, non negative (for example negative exponentials of $|x|$). Then, by using $|r|\le M_1 \xi + \beta_\xi(r)$, we have
\begin{align}
  \E \Big[\int_{\R^d}\int_{\R^d} & \big| u(t,x)-v(t,y)\big|\varrho_\delta(x-y)\Phi(x)\,dx\,dy \Big] - C\xi \|\Phi\|_{L^1(\R^d} 
 \label{inq:pre-final}
 \\
 \le& e^{Ct} \E\Big[\int_{\R^d}\int_{ \R^d}\big| u_0(x) -v_0(y)\big|\varrho_\delta(x-y) \,dx\, dy\Big] 
+  Ce^{Ct}\frac{(\|A'-B'\|_\infty+\xi)}{\delta} t \notag \\
& \qquad \quad  + Ce^{Ct} \Bigg(\frac{\mathcal{E}(\sigma, \widetilde{\sigma})^2}{\xi} + \xi + ||f^\prime-g^\prime||_{\infty} + \sqrt{\mathcal{D}(\eta, \widetilde{\eta})} + \frac{\mathcal{D}(\eta, \widetilde{\eta})}{\xi} \Bigg) t.\notag
 \end{align}
Again, in view of BV bound of the entropy solutions $u(t,x)$ and $v(t,y)$, we have 
 \begin{align}
\E \Big[\int_{\R^d} & \big|v(t,y)-u(t,y)\big|\,dy\Big]  \notag \\
  \le &   \E \Big[\int_{\R^d}\int_{\R^d}  \big| v(t,y)-u(t,x)\big|
\varrho_\delta(x-y)\,dx\,dy \Big]
 + \E \Big[\int_{\R^d}\int_{\R^d} \big| u(t,x) -u(t,y)\big|\varrho_\delta(x-y)\,dx\,dy \Big]\notag \\
\le & \E \Big[ \int_{\R^d}\int_{\R^d} \big|v(t,y) -u(t,x)\big|\varrho_\delta(x-y)\,dx\,dy \Big]
  + \delta\, \E\Big[|u_0|_{BV(\R^d)}\Big], \label{estimate:solu}
 \end{align} 
and 
 \begin{align}
   \E \Big[\int_{\R^d}\int_{ \R^d} \big| u_0(x) -v_0(y)\big|\varrho_\delta(x-y) \,dx\, dy\Big]
   \le  \E \Big[\int_{\R^d} \big|u_0(x)-v_0(x)\big| \,dx \Big]+ \delta\,  \E\Big[|v_0|_{BV(\R^d)}\Big]. \label{estimate:ini}
 \end{align}
Thus, thanks to \eqref{estimate:solu} and \eqref{estimate:ini}, we obtain from \eqref{inq:pre-final} 
\begin{align}
   \E \Big[\int_{\R^d}\big| u(t,x)-v(t,x)\big|\Phi(x)\,dx \Big]
  & \le e^{Ct} \E\Big[\int_{\R^d}\big| u_0(x) -v_0(x)\big| \,dx\Big] + C \big(\xi +\delta\big) +  Ce^{Ct}\frac{(\|A'-B'\|_\infty+\xi)}{\delta}t \notag \\
  & \quad \quad  + Ce^{Ct} \Bigg(\frac{\mathcal{E}(\sigma, \widetilde{\sigma})^2}{\xi} + \xi 
  + ||f^\prime-g^\prime||_{\infty} + \sqrt{\mathcal{D}(\eta, \widetilde{\eta})} 
  + \frac{\mathcal{D}(\eta, \widetilde{\eta})}{\xi} \Bigg) t. \label{esti:pre-final-1}
  \end{align}
By choosing $\xi= \max \bigg\{\mathcal{E}(\sigma, \widetilde{\sigma}), \sqrt{\mathcal{D}(\eta,\widetilde{\eta})}\bigg\}\sqrt{t}$ 
in \eqref{esti:pre-final-1}, we arrive at 
 \begin{align}
   & \E \Big[\int_{\R^d}\big| u(t,x)-v(t,x)\big|\Phi(x)\,dx \Big] \notag \\
  & \le e^{Ct} \,\E\Big[\int_{\R^d}\big| u_0(x) -v_0(x)\big| \,dx\Big] + Ce^{Ct} 
  \Bigg( \max \bigg\{ \mathcal{E}(\sigma, \widetilde{\sigma}), \sqrt{\mathcal{D}(\eta,\widetilde{\eta})}\bigg\} 
  \sqrt{t}(1+t) +\delta +  ||f^\prime-g^\prime||_{\infty} t \Bigg) \notag \\
 &  \quad  +  \frac{Ce^{Ct}}{\delta}\, \Bigg(\|A'-B'\|_\infty t+ \max \bigg\{ \mathcal{E}(\sigma, \widetilde{\sigma}), 
 \sqrt{\mathcal{D}(\eta,\widetilde{\eta})}\bigg\}t\sqrt{t} \Bigg)
 \\  & \le e^{Ct} \,\E\Big[\int_{\R^d}\big| u_0(x) -v_0(x)\big| \,dx\Big] + C_Te^{Ct}
 \Bigg( \max \bigg\{ \mathcal{E}(\sigma, \widetilde{\sigma}), \sqrt{\mathcal{D}(\eta,\widetilde{\eta})}\bigg\} \sqrt{t} +\delta +  ||f^\prime-g^\prime||_{\infty} t \Bigg) \notag \\
 &  \quad  +  \frac{C_Te^{Ct}}{\delta}\, \Bigg(\|A'-B'\|_\infty t+ \max \bigg\{ \mathcal{E}(\sigma, \widetilde{\sigma}), \sqrt{\mathcal{D}(\eta,\widetilde{\eta})}\bigg\}t \Bigg).
\label{esti:pre-final-2}
  \end{align}
Now we simply choose $\delta^2 = \max\bigg\{ \|A'-B'\|_\infty,  \mathcal{E}(\sigma, \widetilde{\sigma}), \sqrt{\mathcal{D}(\eta,\widetilde{\eta})}\bigg\}t$ in \eqref{esti:pre-final-2} and conclude that for a.e. $t>0$, 
\begin{align*}
     & \E \Big[\int_{\R^d}\big| u(t,x)-v(t,x)\big|\Phi(x)\,dx \Big] \notag \\
     & \quad \le C_Te^{Ct} \Bigg\{ \E\Big[\int_{\R^d}\big| u_0(x) -v_0(x)\big| \,dx\Big] 
     + \max \bigg\{ \mathcal{E}(\sigma, \widetilde{\sigma}),\sqrt{\mathcal{D}(\eta,\widetilde{\eta})}\bigg\} \sqrt{t} + ||f^\prime-g^\prime||_{\infty} t  \\
      & \hspace{3cm}+ \max\bigg\{ \sqrt{\|A'-B'\|_\infty}, \, \sqrt{\mathcal{E}(\sigma, \widetilde{\sigma})}, \sqrt[4]{\mathcal{D}(\eta,\widetilde{\eta})}  \bigg\}\sqrt{t} \Bigg\},
\end{align*}
 for some constant $C_T$ depending on $T$, $|u_0|_{BV(\R^d)}, |v_0|_{BV(\R^d)}, \|f^{\prime\prime}\|_{\infty}, \|f^\prime\|_{\infty}$, $\|\Phi\|_{L^1}$ and $\|B^\prime\|_{\infty}$.
 This completes the proof.

\section{Proof of the Main Corollary}
\label{sec:cor}
It is already known (cf. \cite{BaVaWit_2014,BisMajVal}) that the vanishing viscosity solutions $u_\eps(t,x)$ of the problem \eqref{eq:viscous-Brown} converge (in an appropriate sense) to the unique entropy solution $u(t,x)$ of the stochastic conservation law \eqref{eq:stoc_con_brown}. However, the nature of such convergence described by a rate of convergence is not available. As a by product of the Main Theorem, we explicitly obtain the rate of convergence of vanishing viscosity solutions to the unique
BV entropy solution of the underlying problem \eqref{eq:stoc_con_brown}. 

For $\eps>0$, let $u_\eps$ be a weak solution to the problem \eqref{eq:viscous} with data $(A,f,\sigma, \eta,u_0)$ and $u_\theta$ be a weak solution to the viscous problem 
\eqref{eq:viscous-Brown} with small positive parameter $\theta$ which is different from $\eps$.
A similar arguments as in the proof of the Main Theorem yields, thanks to \eqref{esti:final-01},  
\begin{align}
 0 \le &  \E \Big[\int_{\R^d} \int_{\R^d} \big|u_0(x)-u_0(y)\big|\psi(0,y)\varrho_{\delta} (x-y)\,dx\,dy\Big] \label{esti:final-02}  
 \\& + 
 \E \Big[\int_{\Pi_T}\int_{\R^d} \beta_\xi \big(u_\theta(t,x)-u_\eps(t,y)\big) 
[ \partial_t\psi(t,y) + C \psi(t,y)]
 \varrho_\delta(x-y)\,dy\,dx\,dt\Big]\notag 
 \\&
 + C  \xi  \int_{t=0}^T ||\psi(t,\cdot)||_{L^\infty(\R^d)}\,dt \notag
+  C\frac{\xi}{\delta} \int_{t=0}^T \|\psi(t,\cdot)\|_{L^\infty(\R^d)} \,dt+ \frac{C}{\delta} \big( \eps +\theta\big) \notag
\\&
- \E \Big[\int_{\Pi_T}\int_{\R^d} f^{\beta_\xi} \big(u_\theta(t,x),u_\eps(t,y)\big) \cdot \grad_y \psi(t,y)\varrho_{\delta}(x-y)\,dx\,dy\,dt \Big]   \notag
\\ &-\E \Big[\int_{\Pi_T}\int_{\R^d} \beta^\prime_\xi(u_\theta(t,x)-u_\eps(t,y))[\nabla_x A(u_\theta(t,x)) -\nabla_y A(u_\eps(t,y))]  \varrho_{\delta}(x-y) \nabla_y\psi(t,y)\,dx\,dy\,dt \Big].\notag
\end{align}
 Let the family  $\{u_\theta(s,x)\}_{\theta>0}$ converges to the unique entropy solution $u(s,x)$ as $\theta \goto 0$. Thus, passing to the limit as $\theta \goto 0$ in 
 \eqref{esti:final-02}, we have
 \begin{align}
 0 \le &  \E \Big[\int_{\R^d} \int_{\R^d} \big|u_0(x)-u_0(y)\big|\psi(0,y)\varrho_{\delta} (x-y)\,dx\,dy\Big] \label{esti:final-02bis}  
 \\& + 
 \E \Big[\int_{\Pi_T}\int_{\R^d} \beta_\xi \big(u(t,x)-u_\eps(t,y)\big)  
[ \partial_t\psi(t,y) + C \psi(t,y)]
\varrho_\delta(x-y)\,dy\,dx\,dt\Big]\notag 
 \\&
 + C  \xi  \int_{t=0}^T ||\psi(t,\cdot)||_{L^\infty(\R^d)}\,dt \notag
+  C\frac{\xi}{\delta} \int_{t=0}^T \|\psi(t,\cdot)\|_{L^\infty(\R^d)} \,dt + \frac{C\eps}{\delta}  \notag
\\&
- \E\Big[\int_{\Pi_T}\int_{\R^d} f^{\beta_\xi} \big(u(t,x),u_\eps(t,y)\big) \cdot \grad_y \psi(t,y)\varrho_{\delta}(x-y)\,dx\,dy\,dt \Big]   \notag
\\ &-\E\Big[\int_{\Pi_T}\int_{\R^d} \beta^\prime_\xi(u_\theta(t,x)-u_\eps(t,y))[\nabla_x A(u(t,x)) - \nabla_y A(u_\eps(t,y))]  \varrho_{\delta}(x-y) \nabla_y\psi(t,y)\,dx\,dy\,dt \Big].\notag
\end{align}
As before, we use $\psi(s,x)=\psi_h^t(s)\zeta(x)$ where $\psi_h^t(s)$ and $\zeta(x)$ are described previously and then pass to the limit as $h\goto 0$. Again, sending $\zeta \goto {\bf 1}_{\R^d}$, the resulting expression reads as
\begin{align*}
\E \Big[\int_{\R^{2d}} & \beta_\xi \big(u(t,x)-u_\eps(t,y)\big) \varrho_\delta(x-y)\,dy\,dx\Big]
 \\ \le &  
\int_0^t\E \Big[\int_{\R^{2d}} \beta_\xi \big(u(s,x)-u_\eps(s,y)\big) \varrho_\delta(x-y)\,dy\,dx\Big]\,ds \notag
\\&+ 
 \E \Big[\int_{\R^d} \int_{\R^d} \big|u_0(x)-u_0(y)\big|\varrho_{\delta} (x-y)\,dx\,dy\Big] 
+  C  \xi  t 
+  C\frac{\xi}{\delta}  + \frac{C\eps}{\delta} .\notag
\end{align*}
and
\begin{align}
\E \Big[\int_{\R^{2d}} & \beta_\xi \big(u(t,x)-u_\eps(t,y)\big) \varrho_\delta(x-y)\,dy\,dx\Big] \label{esti:final-02ter}  
 \\ \le &  
e^{Ct}\Big( 
 \E \Big[\int_{\R^d} \int_{\R^d} \big|u_0(x)-u_0(y)\big|\varrho_{\delta} (x-y)\,dx\,dy\Big] 
+  C  \xi  t 
+  C\frac{\xi}{\delta}  + \frac{C\eps}{\delta} \Big).\notag
\end{align}
And finally, passing to the limit with respect to $\xi$ yields 
  \begin{align}
 \E \Big[\int_{\R^{2d}}  \big|u(t,x)-u_\eps(t,y)\big| \varrho_\delta(x-y)\,dy\,dx\Big]  
\le 
e^{Ct}\Big( \E \Big[\int_{\R^d} \int_{\R^d} \big|u_0(x)-u_0(y)\big|\varrho_{\delta} (x-y)\,dx\,dy\Big] 
  + \frac{C\eps}{\delta}\Big).
  \label{esti:final-02qua} 
\end{align}
Again, since $u_\eps(t,y)$ and $u(t,x)$ satisfy spatial BV bound, bounded by the BV norm of $u_0(\cdot)$, we obtain 
  \begin{align}
  & \E \Big[\int_{\R^d}\big| u_\eps(t,x)-u(t,x)\big|\,dx\Big]\le Ce^{Ct}(\delta +  \frac{\eps}{\delta}) .\label{esti:final}
  \end{align}
Finally, choosing the optimal value of $\delta$ in \eqref{esti:final} yields:  for a.e. $t>0$, 
  \begin{align*}
  \E\Big[\|u_\eps(t,\cdot)-u(t,\cdot)\|_{L^1(\R^d)}\Big] \le C e^{Ct}\eps^\frac{1}{2},
  \end{align*}
where $C>0$ is a constant depending only on  
$|u_0|_{BV(\R^d)}, \|f^{\prime\prime}\|_{\infty}, \|f^\prime\|_{\infty}$, and $\|A^\prime\|_{\infty}$. 
This completes the proof.
  

\section{Fractional BV Estimates}
\label{sec:frac}
In this section, we consider a more general class of stochastic balance laws driven by L\'{evy} noise of the type
 \begin{equation}
 \label{eq:levy_stochconservation_laws_spatial}
 \begin{cases} 
 du(t,x)- \mbox{div}_x f(u(t,x)) \,dt-\Delta_x A(u(t,x))\,dt  \\
\qquad \qquad \qquad \qquad \qquad \qquad \qquad  =\sigma(x,u(t,x))\,dW(t) + \int_{|z|> 0} \eta( x,u(t,x);z) \, \widetilde{N}(dz,dt), & x \in \Pi_T, \\
 u(0,x) = u_0(x), & x\in \R^d,
\end{cases}
\end{equation}
Observe that, noise coefficients $\sigma(x,u)$ and $\eta(x,u;z)$ depend explicitly on the spatial position $x$. Moreover, we assume that $\sigma(x,u)$, and 
$\eta(x,u;z)$ satisfy the following assumptions:
\begin{Assumptions2} 
\item \label{B31} There exists a positive constant $K_1 > 0$  such that 
\begin{align*} 
\big| \sigma(x,u)-\sigma(y,v)\big|  \leq K_1 \Big(|x-y| + |u-v|\Big), ~\text{for all} \,\,u,v \in \R; ~~ x,y \in \R^d.
\end{align*}  
Moreover, we assume that $\sigma(x,0)=0$, for all $x \in \R^d$. As a consequence,
$|\sigma(x,u)| \le K_1 |u|$.
 \item \label{B3}  There exist positive constants $K_2>0$  and  $\lambda^* \in (0,1)$  such that 
 \begin{align*}
 | \eta(x,u;z)-\eta(y,v;z)|  \leq  (\lambda^* |u-v| + K_2|x-y|)( |z|\wedge 1), ~\text{for all}~  u,v \in \R;~~z\in \R 
 ;~~ x,y \in \R^d.
 \end{align*}
Moreover, we assume that $\eta(x,0;z)=0$, for all $x \in \R^d$, and $z\in \R$. In particular, this implies that
\begin{align*}
|\eta(x,u;z)|  \leq  \lambda^* |u|( |z|\wedge 1),\quad \mbox{and}
\quad |u| \leq \frac{1}{1-\alpha\lambda^*} |u + \alpha \eta(x,u;z) |,\ \forall \alpha\in[0,1].
\end{align*}
\item \label{B4} There exists a non-negative function  $ g(x)\in L^\infty(\R^d)\cap L^2(\R^d)$  such that 
\begin{align*}
|\eta(x,u;z)| \le g(x)(1+|u|)(|z|\wedge 1), ~\text{for all}~ (x,u,z)\in \R^d \times  \R\times \R.  
\end{align*} 
\item \label{B5} $ A:\R\rightarrow \R$ is a non-decreasing  Lipschitz continuous function with $A(0)=0$. 
Moreover, $t\longmapsto \sqrt{A^\prime(t)}$ has a modulus of continuity $\omega_A$
such that $\frac{\omega_A(r)}{r^{\frac{2}{3}}} \longrightarrow 0$ as $r \goto 0$.
\end{Assumptions2}
Clearly, our continuous dependence estimate is not applicable for problems of type  \eqref{eq:levy_stochconservation_laws_spatial} due to the non-availability of BV estimate for the solution of \eqref{eq:levy_stochconservation_laws_spatial}. We refer to \cite[Section $2$]{Chen-karlsen_2012} for a discussion on this point in case of diffusion driven balance laws. However, it is possible to obtain a fractional BV estimate. 
To that context, drawing primary motivation from the discussions in \cite{Chen-karlsen_2012}, we intend to show that a uniform fractional BV estimate can be obtained for the solution of the regularized stochastic parabolic problem given by
\begin{align} 
\label{eq:stability-3}
&du_{\eps}(t,x)- \mbox{div}_x f(u_{\eps}(t,x)) \,dt-\Delta_x A(u_{\eps}(t,x))\,dt  \\
&\qquad \qquad \qquad \qquad \qquad =\sigma(x,u_{\eps}(t,x))\,dW(t) +  
  \int_{|z|> 0} \eta(x,u_\eps(t,x);z)\widetilde{N}(dz, dt)+ \eps \Delta_{xx} u_\eps(t,x) \,dt \notag
\end{align} 
Regarding equation \eqref{eq:stability-3}, we mention that existence and regularity 
of the solution to the problem \eqref{eq:stability-3} has been studied in \cite{BaVaWit_2014,BisMajVal}.
We start with a deterministic lemma, related to the estimation of the modulus of continuity
of a given integrable function, and also an useful link between Sobolev and Besov spaces. In fact, we have the following lemma, a proof of which can be found in \cite[Lemma $2$]{Chen-karlsen_2012}.
  \begin{lem}\label{lem: deterministic-modulus-continuity}
    Let $h:\R^d\goto \R$ be a given integrable function, $ 0\le \zeta \in C_c^\infty(\R^d)$ and \{$J_\delta\}_{\delta>0}$ be
     a sequence of symmetric mollifiers, i.e., $J_\delta(x)=\frac{1}{\delta^d} J(\frac{|x|}{\delta}),\, 0\le J \in C_c^\infty(\R)$,
    $\mbox{supp}(J) \subset [-1,1],\, J(-\cdot)= J(\cdot)$ and $\int J=1$. Then
    \begin{itemize}
     \item [(a)] For $r,s \in (0,1)$ with $r<s$, there exists a finite
    constant $C_1=C_1(J,d,r,s)$ such that
    \begin{align}
    \int_{\R^d}\int_{\R^d} | h(x+z)-h(x-z)| & J_\delta(z)\zeta(x)\,dx\,dz \notag \\
    \le & C_1\,\delta^r \sup_{|z|\le \delta} |z|^{-s}\int_{\R^d} |h(x+z)-h(x-z)|\zeta(x)\,dx.\label{lem:modulus-continuity-part-1}
    \end{align}
    \item[(b)] For $r,s \in (0,1)$ with $r<s$, there exists a finite
    constant $C_2=C_2(J,d,r,s)$ such that
    \begin{align}
   \sup_{|z|\le \delta}\int_{\R^d} & |h(x+z)-h(x)|\zeta(x)\,dx \notag \\
     & \quad \le C_2 \delta^r \sup_{0<\delta\le 1} \delta^{-s}  \int_{\R^d}\int_{\R^d} | h(x+z)-h(x-z)| J_\delta(z)\zeta(x)\,dx\,dz 
     + C_2 \delta^r ||h||_{L^1(\R^d)}.\label{lem:modulus-continuity-part-2}
    \end{align} 
 \end{itemize}
  \end{lem}
  
Now we are in a position to state and prove a theorem regarding fractional BV estimation of solutions of \eqref{eq:stability-3}.
\begin{thm}[Fractional BV estimate]
  \label{thm:fractional BV estimate}
Let the assumptions ~\ref{A1}, ~\ref{A3}, ~\ref{B31}, ~\ref{B3}, ~\ref{B4} and ~\ref{B5} hold. Let $u_\eps$ be a solution 
of \eqref{eq:stability-3} with the initial data $u_0(x)$ belongs to the Besov space $B^{\mu}_{1, \infty} (\R^d)$ for 
some $\mu \in (\frac27,1)$. 
Moreover, we assume that $f^{\prime\prime} \in L^\infty$. Then, for fixed $T>0$ and $R>0$, there exits a constant $C(T,R)$, 
independent of $\eps$, such that for any $0< t <T$,
 \begin{align*}
  \sup_{|y| \le \delta} \E\Big[\int_{K_R} \big|u_\eps(t,x+y)-u_\eps(t,x)\big|\,dx\Big] \le C(T,R)\,\delta^r,
 \end{align*} 
for some $r \in (0,\frac{2}{7})$ and $K_{R}:=\{x: |x| \le R\}$.
\end{thm}
  
\begin{proof}
Let $\zeta\in \mathcal{K}:= \{\zeta  \in C^2(\R^d)\cap L^1(\R^d)\cap L^{\infty}(\R^d): |\grad \zeta| \le C \zeta, |\Delta \zeta| \le C \zeta\}$ be any function. Then by Lemma~\ref{imp}, there exists a sequence of functions 
$\lbrace \zeta_R \rbrace \subset C_c^{\infty}(\R^d)$ such that, in particular, $\zeta_R \mapsto \zeta$ pointwise.
Let $J_\delta$ be a sequence of mollifier in $\R^d$ as mentioned in Lemma \ref{lem: deterministic-modulus-continuity}.
Consider the test function 
$$\psi_{\delta}^R(x,y): =J_\delta \left(\frac{x-y}{2} \right)\, \zeta_R \left(\frac{x+y}{2} \right).$$ 
In the sequel, with a slight abuse of notations, we denote $\psi_\delta=\psi_\delta^R$ and $\zeta=\zeta_R$.
Subtracting two solutions $u_\eps(t,x)$, $u_\eps(t,y)$ of \eqref{eq:stability-3}, 
and applying It\^{o}-L\'{e}vy formula to that resulting equations, we obtain (cf. \cite{Chen-karlsen_2012})
\begin{align}
 &\E \Big[\int_{\R^d}\int_{\R^d} \beta_\xi \big( u_\eps(t,x) -u_\eps(t,y)\big) \psi_\delta(x,y) \,dx \,dy \Big]
- \E \Big[\int_{\R^d}\int_{\R^d} \beta_\xi\big( u_\eps(0,x) -u_\eps(0,y)\big) \psi_\delta(x,y) \,dx\, dy \Big] \notag \\
& \le \E \Big[\int_{0}^t \int_{\R^d}\int_{\R^d} f^{\beta} \big(u_\eps(s,x),u_\eps(s,y)\big)\cdot \grad \zeta(\frac{x+y}{2}) J_\delta(\frac{x-y}{2})\,dx\,dy\,ds \Big] \notag \\
  & + \E \Big[\int_{0}^t \int_{\R^d}\int_{\R^d} \Big(f^{\beta}\big(u_\eps(s,y),u_\eps(s,x)\big)- f^{\beta} \big(u_\eps(s,x),u_\eps(s,y)\big)\Big)\cdot \grad_y \psi_\delta(x,y)\,dx\,dy\,ds \Big]\notag \\
   &
- \E \Big[ \int_{0}^t \int_{\R^d}\int_{\R^d} \beta_\xi^{\prime\prime}\big(u_\eps(s,x)-u_\eps(s,y)\big)
  \Big(|\grad_y G(u_\eps(s,y))|^2 + |\grad_x G(u_\eps(s,x)) |^2\Big)\,\psi_\delta(x,y)\,dx\,dy\,ds  \Big] \notag
 \\ &+ \E \Big[\int_{0}^t \int_{\R^d}\int_{\R^d} \big(A^{\beta} (u_\eps(s,y), u_\eps(s,x)) \Delta_{y} \psi_\delta(x,y) + A^{\beta} (u_\eps(s,x), u_\eps(s,y)) \Delta_{x} \psi_\delta(x,y) \big) \,dx\,dy\,ds  \Big] \notag \\
  &+ \E \Big[ \int_{0}^t \int_{\R^d}\int_{\R^d} 
  \eps\,\beta_\xi\big( u_\eps(r,x) -u_\eps(r,y)\big) J_\delta(\frac{x-y}{2}) \Delta \zeta(\frac{x+y}{2})\, dx\, dy\, dr \Big]\notag \\
  & + \frac12 \E \Big[ \int_{0}^t \int_{\R^d}\int_{\R^d} 
 \beta_\xi^{\prime \prime} \big( u_\eps(r,x) -u_\eps(r,y)\big) \Big( \sigma(x, u_\eps(r,x) - \sigma(y, u_\eps(r,y) \Big)^2 
\psi_\delta(x,y) \,dx \,dy \, dr \Big] \notag \\
 & + \E \Big[\int_{0}^t\int_{|z| > 0} \int_{\R^d}\int_{\R^d} \int_{\rho=0}^1 
 \beta_\xi^{\prime \prime}\Big( u_\eps(r,x) -u_\eps(r,y)+\rho\big(\eta(x,u_\eps(r,x);z)-\eta(y,u_\eps(r,y);z)\big)\Big)\notag \\
&\hspace{3cm}\times \big| \eta(x,u_\eps(r,x);z)
-\eta(y,u_\eps(r,y);z)\big|^2   \psi_\delta(x,y)\,d\rho\,dx \,dy\,m(dz) \,dr \Big]. \label{test-1}
\end{align} 
To this end, we see that
\begin{align*}
\grad_{x} \psi_\delta(x,y) + \grad_{y} \psi_\delta(x,y) &= 2\, \grad \zeta \left(\frac{x+y}{2}\right) J_\delta\left(\frac{x-y}{2}\right), \\
\Big| A^\beta(u,v)-A^\beta(v,u)\Big| &\le C ||A^{\prime\prime}||_{\infty}\,\xi |u-v|.
\end{align*}
Moreover, a similar analysis as in 
Lemma~\ref{estim_A1_3}-\eqref{esti:a1+a2+a3-final} reveals that
\begin{align*}
&\int_{0}^t \int_{\R^d}\int_{\R^d} \big(A^{\beta} (u_\eps(s,y), u_\eps(s,x)) \Delta_{y} \psi_\delta(x,y)
+ A^{\beta} (u_\eps(s,x), u_\eps(s,y)) \Delta_{x} \psi_\delta(x,y) \big) \,dx\,dy\,ds \\
& -  \int_{0}^t \int_{\R^d}\int_{\R^d} \beta_\xi^{\prime\prime}\big(u_\eps(s,x)-u_\eps(s,y)\big)
  \Big(|\grad_y G(u_\eps(s,y))|^2 + |\grad_x G(u_\eps(s,x)) |^2\Big)\,\psi_\delta(x,y)\,dx\,dy\,ds 
  \\
& \le C\frac{\xi^{4/3}}{\delta^2}\| \zeta\|_{L^\infty(\R^d)} t + C \frac{\xi^{4/3}}{\delta}\| \grad\zeta\|_{L^\infty(\R^d)} t\\
&\quad -  \int_0^t \int_{\R^{2d}} \beta^\prime_\xi(u_\eps(s,x)-u_\eps(s,y))[\nabla_x A(u_\eps(s,x))
- \nabla_y A(u_\eps(s,y))]  J_\delta(\frac{x-y}{2}) \nabla \zeta(\frac{x+y}{2})\,dx\,dy\,ds \\
& \le C\frac{\xi^{4/3}}{\delta^2}\| \zeta\|_{L^\infty(\R^d)} t  +C \frac{\xi^{4/3}}{\delta}\| \grad\zeta\|_{L^\infty(\R^d)} t \\
& + \int_{0}^t \int_{\R^d}\int_{\R^d} A^\beta \big(u_\eps(s,x),u_\eps(s,y)\big)\cdot 
\Delta \zeta(\frac{x+y}{2}) J_\delta(\frac{x-y}{2})\,dx\,dy\,ds \\
& + \int_{0}^t \int_{\R^d}\int_{\R^d}\Big(A^\beta\big(u_\eps(s,y),u_\eps(s,x)\big)- A^\beta \big(u_\eps(s,x),u_\eps(s,y)\big)\Big)
\cdot \grad_y \bigg(\grad \zeta(\frac{x+y}{2}) J_\delta(\frac{x-y}{2})\bigg)\,dx\,dy\,ds,
\end{align*} 
where, in view of \ref{B5}, we have used the fact that (cf. Lemma~\ref{estim_A1_3})
\begin{align*}
&\mathcal{A}_{1} +  \mathcal{A}_{2,1} + \mathcal{A}_{3,1} 
\\ \leq &
 \Bigg| \E\Big[\int_{\Pi_t}\int_{\R^d} \Bigg\{\int_{v_\eps(t,y)}^{u_\theta(t,x)}\beta_\xi^{\prime\prime}\big( \tau-v_{\eps}(t,y)\big) 
 \big[\sqrt{A'(\tau)}- \sqrt{A'(v_\eps(t,y))}\big]^2\,d\tau\Bigg\} \\
 & \hspace{8cm} \times \grad_y v_\eps(t,y)
  \nabla_x J_{\delta}(\frac{x-y}{2}) \zeta(\frac{x+y}{2}) \,dx\,dy\,dt \Big] \Bigg| 
\\ = &
 \Bigg| \E\Big[\int_{\Pi_t}\int_{\R^d} \Bigg\{\nabla_y \int^{v_\eps(t,y)}_{u_\theta(t,x)} \int_{\sigma}^{u_\theta(t,x)}\beta_\xi^{\prime\prime}\big( \tau-\sigma\big)  \big[\sqrt{A'(\tau)}- \sqrt{A'(\sigma)}\big]^2\,d\tau\,d\sigma\Bigg\} \\
 & \hspace{8cm} 
  \times \nabla_x J_{\delta}(\frac{x-y}{2}) \zeta(\frac{x+y}{2}) \,dx\,dy\,dt \Big] \Bigg| 
\\ = &
 \Bigg| \E\Big[\int_{\Pi_t}\int_{\R^d} \Bigg\{\int^{v_\eps(t,y)}_{u_\theta(t,x)} \int_{\sigma}^{u_\theta(t,x)}\beta_\xi^{\prime\prime}\big( \tau-\sigma\big)  \big[\sqrt{A'(\tau)}- \sqrt{A'(\sigma)}\big]^2\,d\tau\,d\sigma\Bigg\} \\
 & \hspace{8cm} 
 \mathrm{div}_y\bigg[ \nabla_x J_{\delta}(\frac{x-y}{2}) \zeta(\frac{x+y}{2})\bigg] \,dx\,dy\,dt \Big] \Bigg| 
\\ \leq &
C \xi^{4/3} \E\Big[\int_{\Pi_t}\int_{\R^d} |v_\eps(t,y)-u_\theta(t,x)| |
 \mathrm{div}_y\bigg[ \nabla_x J_{\delta}(\frac{x-y}{2}) \zeta(\frac{x+y}{2})\bigg]| \,dx\,dy\,dt \Big] 
\\ \leq &
C \frac{\xi^{4/3}}{\delta^2}\| \zeta\|_{L^\infty(\R^d)} t + C \frac{\xi^{4/3}}{\delta}\| \grad\zeta\|_{L^\infty(\R^d)} t
\end{align*}
At this point we let $R\mapsto \infty$ in the test function $\zeta=\zeta_R$. Moreover, keeping in mind that for any function $\zeta \in \mathcal{K}$
satisfies $|\nabla \zeta(x)| \le C \zeta(x)$, and $|\Delta \zeta(x)| \le C \zeta(x)$, we have from \eqref{test-1} 
\begin{align}
& \E \Big[\int_{\R^d}\int_{\R^d} \beta_\xi\big( u_\eps(t,x) -u_\eps(t,y)\big) \psi_\delta(x,y) \,dx \,dy\Big]
- \E \Big[\int_{\R^d}\int_{\R^d} \beta_{\xi}\big( u_\eps(0,x) -u_\eps(0,y)\big) \psi_\delta(x,y) \,dx\, dy\Big] \notag \\
& \le  C \frac{\xi^{4/3}}{\delta^2}\| \zeta\|_{L^\infty(\R^d)} t + C \frac{\xi^{4/3}}{\delta}\|\zeta\|_{L^\infty(\R^d)} t \notag \\
& + C (||f^\prime||_{\infty} + ||A^{\prime}||_{\infty})\int_{0}^t\,\E \Big[\int_{\R^d}\int_{\R^d}  \big|u_\eps(s,x)-u_\eps(s,y)\big| 
\zeta(\frac{x+y}{2}) J_\delta(\frac{x-y}{2})\,dx\,dy\Big]\,ds  \notag \\
& +C ||f^{\prime\prime}||_{\infty} \xi\,\E\Big[\int_{0}^t \int_{\R^d}\int_{\R^d} \big|u_\eps(s,x)-u_\eps(s,y)\big| \zeta(\frac{x+y}{2})
 J_\delta(\frac{x-y}{2})\,dx\,dy\,ds\Big] \notag \\
 & + C ||f^{\prime\prime}||_{\infty} \xi\,\E \Big[\int_{0}^t \int_{\R^d}\int_{\R^d} \big|u_\eps(s,x)-u_\eps(s,y)\big| \zeta(\frac{x+y}{2})
|\grad_y J_\delta(\frac{x-y}{2})|\,dx\,dy\,ds\Big] \notag \\
& +C ||A^{\prime\prime}||_{\infty} \xi\,\E\Big[\int_{0}^t \int_{\R^d}\int_{\R^d} \big|u_\eps(s,x)-u_\eps(s,y)\big| \zeta(\frac{x+y}{2})
 J_\delta(\frac{x-y}{2})\,dx\,dy\,ds\Big] \notag \\
 & + C ||A^{\prime\prime}||_{\infty} \xi\,\E \Big[\int_{0}^t \int_{\R^d}\int_{\R^d} \big|u_\eps(s,x)-u_\eps(s,y)\big| \zeta(\frac{x+y}{2})
|\grad_y J_\delta(\frac{x-y}{2})|\,dx\,dy\,ds\Big] \notag \\
&+   C\,\eps\, \int_{0}^t \E \Big[ \int_{\R^d}\int_{\R^d} 
  \big| u_\eps(r,x) -u_\eps(r,y)\big| J_\delta(\frac{x-y}{2}) \zeta(\frac{x+y}{2})\, dx\, dy\Big]\, dr \notag \\
  & + \frac12 \E \Big[ \int_{0}^t \int_{\R^d}\int_{\R^d} 
 \beta_\xi^{\prime \prime} \big( u_\eps(r,x) -u_\eps(r,y)\big) \Big( \sigma(x, u_\eps(r,x) - \sigma(y, u_\eps(r,y) \Big)^2 
\psi_\delta(x,y) \,dx \,dy \, dr \Big] \notag \\
 & + \E \Big[\int_{0}^t\int_{|z| > 0} \int_{\R^d}\int_{\R^d} \int_{\rho=0}^1
 \beta_\xi^{\prime \prime}\Big( u_\eps(r,x) -u_\eps(r,y)+\rho\big(\eta(x,u_\eps(r,x);z)-\eta(y,u_\eps(r,y);z)\big)\Big)\notag
\\&\hspace{4.5cm}\times \big| \eta(x,u_\eps(r,x);z)
-\eta(y,u_\eps(r,y);z)\big|^2  \psi_\delta(x,y)\,d\rho\,dx \,dy\,m(dz) \,dr\Big]. \label{eq:bv-frac}
\end{align}

As before, with the help of the uniform $L^1$ estimate \eqref{L1-bound}, we can conclude 
 \begin{align}
  &C  ||f^{\prime\prime}||_{\infty} \xi\,\E \Big[\int_{0}^t \int_{\R^d}\int_{\R^d} \big|u_\eps(s,x)-u_\eps(s,y)\big| \zeta(\frac{x+y}{2})
 J_\delta(\frac{x-y}{2})\,dx\,dy\,ds\Big] \notag \\
 & + C ||f^{\prime\prime}||_{\infty} \xi\,\E \Big[\int_{0}^t \int_{\R^d}\int_{\R^d} \big|u_\eps(s,x)-u_\eps(s,y)\big| \zeta(\frac{x+y}{2})
|\grad_y J_\delta(\frac{x-y}{2})|\,dx\,dy\,ds\Big] \notag \\
& +C ||A^{\prime\prime}||_{\infty} \xi\,\E\Big[\int_{0}^t \int_{\R^d}\int_{\R^d} \big|u_\eps(s,x)-u_\eps(s,y)\big| \zeta(\frac{x+y}{2})
 J_\delta(\frac{x-y}{2})\,dx\,dy\,ds\Big] \notag \\
 & + C ||A^{\prime\prime}||_{\infty} \xi\,\E \Big[\int_{0}^t \int_{\R^d}\int_{\R^d} \big|u_\eps(s,x)-u_\eps(s,y)\big| \zeta(\frac{x+y}{2})
|\grad_y J_\delta(\frac{x-y}{2})|\,dx\,dy\,ds\Big] \notag \\
&  \le  C (||f^{\prime\prime}||_\infty + ||A^{\prime\prime}||_{\infty}) \big( \xi + \frac{\xi}{\delta} \big) ||\zeta||_{L^{\infty}(\R^d)} t. \label{esti:middle term}
 \end{align}
Next, for the last two terms of \eqref{eq:bv-frac}, we follow the estimates given in \cite{BaVaWit_2014,BisMajVal}, to conclude 
\begin{align}
 &  \frac12 \E \Big[ \int_{0}^t \int_{\R^d}\int_{\R^d} 
 \beta_\xi^{\prime \prime} \big( u_\eps(r,x) -u_\eps(r,y)\big) \Big( \sigma(x, u_\eps(r,x) - \sigma(y, u_\eps(r,y) \Big)^2 
\psi_\delta(x,y) \,dx \,dy \, dr \Big] \notag 
\\ \leq &
\E \Big[ \int_{0}^t \int_{\R^d}\int_{\R^d} 
\beta_\xi^{\prime \prime} \big( u_\eps(r,x) -u_\eps(r,y)\big) \Big( \sigma(x, u_\eps(r,x) - \sigma(x, u_\eps(r,y) \Big)^2 
\psi_\delta(x,y) \,dx \,dy \, dr \Big] \notag 
\\&+
 \E \Big[ \int_{0}^t \int_{\R^d}\int_{\R^d} 
 \beta_\xi^{\prime \prime} \big( u_\eps(r,x) -u_\eps(r,y)\big) \Big( \sigma(x, u_\eps(r,y) - \sigma(y, u_\eps(r,y) \Big)^2 
\psi_\delta(x,y) \,dx \,dy \, dr \Big] \notag
\\ \leq &
C\E \Big[ \int_{0}^t \int_{\R^d}\int_{\R^d} 
\beta_\xi^{\prime \prime} \big( u_\eps(r,x) -u_\eps(r,y)\big) |u_\eps(r,x) - u_\eps(r,y) |^2 
\psi_\delta(x,y) \,dx \,dy \, dr \Big] \notag 
\\&+
C \E \Big[ \int_{0}^t \int_{\R^d}\int_{\R^d} 
 \beta_\xi^{\prime \prime} \big( u_\eps(r,x) -u_\eps(r,y)\big) |y - x|^2 
\psi_\delta(x,y) \,dx \,dy \, dr \Big] \notag
\\ \leq &
C\E \Big[ \int_{0}^t \int_{\R^d}\int_{\R^d} 
\beta_\xi \big( u_\eps(r,x) -u_\eps(r,y)\big) 
\psi_\delta(x,y) \,dx \,dy \, dr \Big] \notag 
+
C\frac{\delta^2}{\xi} \E \Big[ \int_{0}^t \int_{\R^d}\int_{\R^d}  
\psi_\delta(x,y) \,dx \,dy \, dr \Big] \notag
\\ \leq &
C\E \Big[ \int_{0}^t \int_{\R^d}\int_{\R^d} 
\beta_\xi \big( u_\eps(r,x) -u_\eps(r,y)\big) 
\psi_\delta(x,y) \,dx \,dy \, dr \Big] 
+
C\frac{\delta^2}{\xi} t \|\zeta\|_{L^1(\R^d)}
\end{align}
and a similar estimate reveals that 
\begin{align}
 &\E \Big[\int_{0}^t\int_{|z| > 0} \int_{\R^d}\int_{\R^d} \int_{\rho=0}^1
 \beta_\xi^{\prime \prime}\Big( u_\eps(r,x) -u_\eps(r,y)+\rho\big(\eta(x,u_\eps(r,x);z)-\eta(y,u_\eps(r,y);z)\big)\Big)\notag
\\&\hspace{4.5cm}\times \big| \eta(x,u_\eps(r,x);z)
-\eta(y,u_\eps(r,y);z)\big|^2  \psi_\delta(x,y)\,d\rho\,dx \,dy\,m(dz) \,dr\Big] \notag
\\ \leq &
C\E \Big[ \int_{0}^t \int_{\R^d}\int_{\R^d} 
\beta_\xi \big( u_\eps(r,x) -u_\eps(r,y)\big) 
\psi_\delta(x,y) \,dx \,dy \, dr \Big] 
+
C\frac{\delta^2}{\xi} t \|\zeta\|_{L^1(\R^d)}
.\label{esti:last term}
\end{align} 
Now we make use of \eqref{eq:approx to abosx}, \eqref{esti:middle term}, and \eqref{esti:last term} 
 to \eqref{eq:bv-frac} and conclude
\begin{align}
 & \E \Bigg[\int_{\R^d}\int_{\R^d}  \big| u_\eps(t,x) -u_\eps(t,y)\big|J_\delta(\frac{x-y}{2}) \zeta(\frac{x+y}{2}) \,dx \,dy\Bigg] \notag \\
  &\le   \E \Bigg[\int_{\R^d}\int_{\R^d} \big| u_\eps(0,x) -u_\eps(0,y)\big|J_\delta(\frac{x-y}{2}) \zeta(\frac{x+y}{2}) \,dx\, dy\Bigg] 
 + C \frac{\xi^{4/3}}{\delta^2}\| \zeta\|_{L^\infty(\R^d)} t \notag \\
 &
 + C\big(1+||f^\prime||_{\infty}+ ||A^\prime||_{\infty}\big)\int_{0}^t \E \Big[\int_{\R^d}\int_{\R^d} \big|u_\eps(s,x)-u_\eps(s,y)\big| 
\zeta(\frac{x+y}{2}) J_\delta(\frac{x-y}{2})\,dx\,dy\Big]\,ds \notag\\
&+ C (||f^{\prime\prime}||_\infty + ||A^{\prime\prime}||_{\infty}) \big( \xi + \frac{\xi}{\delta} \big) ||\zeta||_{L^{\infty}(\R^d)} t 
 + C( \xi+\frac{\delta^2}{\xi}t)\,||\zeta||_{L^1(\R^d)}.
\end{align} 
A simple application of Gronwall's inequality reveals that
\begin{align}
  & \E \Bigg[\int_{\R^d}\int_{\R^d}  \big| u_\eps(t,x) -u_\eps(t,y)\big|J_\delta(\frac{x-y}{2}) \zeta(\frac{x+y}{2}) \,dx \,dy \Bigg]\notag 
  \\& \le  
\exp^{ t\,C\big(1+ ||f^\prime||_{\infty}+ ||A^\prime||_{\infty}\big)} \E \Big[\int_{\R^d}\int_{\R^d} \big| u_\eps(0,x) -u_\eps(0,y)\big|
J_\delta(\frac{x-y}{2}) \zeta(\frac{x+y}{2}) \,dx\, dy\Big] \notag 
\\& +  
C\exp^{ t\,C\big(1+ ||f^\prime||_{\infty}+ ||A^\prime||_{\infty}\big)}
\Big( \frac{\xi^{4/3}}{\delta^2}\| \zeta\|_{L^\infty(\R^d)} t
+
 (||f^{\prime\prime}||_\infty + ||A^{\prime\prime}||_{\infty}) \big( \xi + \frac{\xi}{\delta} \big) ||\zeta||_{L^{\infty}(\R^d)} t \notag \\
 & \hspace{6.5cm}+ ( \xi+\frac{\delta^2}{\xi}t)\,||\zeta||_{L^1(\R^d)}\Big).
 \label{eq:stochastic-final-8}
\end{align}  
Choosing $\xi = C \delta^\frac{12}{7}$ in \eqref{eq:stochastic-final-8}, we obtain
\begin{align*}
\E \Big[\int_{\R^d}\int_{\R^d}  \big| u_\eps(t,x) & -u_\eps(t,y)\big|J_\delta(\frac{x-y}{2}) \zeta(\frac{x+y}{2})\,dx\,dy\Big] \notag \\
 & \le   C(T) \E \Big[\int_{\R^d}\int_{\R^d}  \big| u_\eps(0,x) -u_\eps(0,y)\big|
J_\delta(\frac{x-y}{2}) \zeta(\frac{x+y}{2}) \,dx\, dy\Big] \notag \\
&   \hspace{4cm} + C(T) \Big( \big(\delta^\frac{2}{7}  +\delta^{5/7}\big) ||\zeta||_{L^\infty(\R^d)} +  \delta^\frac{2}{7}||\zeta||_{L^1(\R^d)} \Big),
\end{align*}  for some constant $C(T)>0$, independent of $\eps$.

Now we make use of the following change of variables 
$$\bar{x}= \frac{x-y}{2}, \,\, \text{and} \,\, \, \bar{y}=\frac{x+y}{2},$$
to rewrite the above inequality (dropping  the bar). The result is
\begin{align}
  \E \Big[\int_{\R^d}\int_{\R^d}  \big| u_\eps(t,x+y) & -u_\eps(t,x-y)\big|J_\delta(y) \zeta(x) \,dx \,dy\Big] \notag \\
& \le   C(T) \E \Big[\int_{\R^d}\int_{\R^d}  \big| u_\eps(0,x+y) -u_\eps(0,x-y)\big|
J_{\delta}(y) \zeta(x) \,dx\, dy\Big]\notag \\
& \hspace{3cm} + C(T) \Big( \big(\delta^\frac{2}{7}  +\delta^{5/7}\big) ||\zeta||_{L^\infty(\R^d)} +  \delta^\frac{2}{7}||\zeta||_{L^1(\R^d)} \Big)
  \label{eq:stochastic-final-9}
\end{align}
In view of \eqref{lem:modulus-continuity-part-2} of the Lemma \ref{lem: deterministic-modulus-continuity}, we obtain
  for $0<r < s<1$
 \begin{align}
 \sup_{|y| \le \delta} \E \Big[\int_{\R^d} & \big|u_\eps(t,x+y)  -u_\eps(t,x)\big|\zeta(x)\,dx\Big] \notag \\
   & \le  C_2 \,\delta^r \sup_{0<\delta \le 1} \delta^{-s}  \E \Big[\int_{\R^d}\int_{\R^d} \big|u_\eps(t,x+y)-u_\eps(t,x-y)\big|
  J_{\delta}(y)\zeta(x)\,dx\,dy\Big]  \notag \\
  &\hspace{7.5cm}+ C_2 \delta^r \E \Big[||u_\eps(t,\cdot)||_{L^1(\R^d)}\Big]. \label{eq:stochastic-final-10}
  \end{align} 
Again, by \eqref{lem:modulus-continuity-part-1} of the Lemma \ref{lem: deterministic-modulus-continuity} and
by \eqref{eq:stochastic-final-9}, we see that for $0<r'<s'<1$
\begin{align}
   \sup_{0<\delta \le 1} \delta^{-s}  \E \Big[\int_{\R^d}\int_{\R^d} & \big|u_\eps(t,x+y)  -u_\eps(t,x-y)\big|
  J_\delta(y)\zeta(x)\,dx\,dy\Big] \notag 
  \\  & \le  
  C(T) \sup_{0<\delta \le 1} \delta^{-s} \E \Big[\int_{\R^d}\int_{\R^d} \big|u_\eps(0,x+y)-u_\eps(0,x-y)\big| J_\delta(y)
  \zeta(x)\,dx\,dy\Big] \notag 
  \\& \hspace{6.5cm} +
   C(T)\delta^{\frac27-s}\Big( ||\zeta||_{L^\infty(\R^d)} + ||\zeta||_{L^1(\R^d)} \Big) \notag \\
    & \le  C(T) \, C_1\, \delta^{-s+r'}\sup_{|y| \le \delta}\Bigg( |y|^{-s'} \, \E\Big[\int_{\R^d}  \big|u_\eps(0,x+y)-u_\eps(0,x)\big|
  \zeta(x)\,dx\Big] \Bigg)  \notag \\
  & \hspace{6.5cm}+ C(T)\delta^{\frac27-s}\Big( ||\zeta||_{L^\infty(\R^d)} + ||\zeta||_{L^1(\R^d)} \Big)  \notag \\
   & \le  C(T)\, \delta^{-s+r'}\E\Big[||u_0||_{B_{1, \infty}^{s'}(\R^d)}\Big] \,||\zeta||_{L^\infty(\R^d)} +
  C(T)\delta^{\frac27-s}\Big( ||\zeta||_{L^\infty(\R^d)} + ||\zeta||_{L^1(\R^d)} \Big).\label{eq:stochastic-final-11}
 \end{align} 
Now we combine \eqref{eq:stochastic-final-10} and \eqref{eq:stochastic-final-11} to obtain
 \begin{align*}
    \sup_{|y| \le \delta}& \E \Big[\int_{\R^d}  \big|u_\eps(t,x+y)-u_\eps(t,x)\big|\zeta(x)\,dx\Big] \notag \\
   &  \le C(T)\,\Bigg[ \delta^{r-s+r'}\E\Big[||u_0||_{B_{1, \infty}^{s'}(\R^d)}\Big] + \delta^{r+\frac27-s}\big( ||\zeta||_{L^\infty(\R^d)} + ||\zeta||_{L^1(\R^d)}\big)\Bigg] 
   + C_2\, \delta^r \E \bigg[||u_\eps(t,\cdot)||_{L^1(\R^d)}\bigg].
 \end{align*}
Setting $r'=s=\frac27$, one gets, for any $r<\frac27$ and $s'>\frac27$,  
  \begin{align*}
    \sup_{|y| \le \delta} \E \Big[\int_{\R^d} & \big|u_\eps(t,x+y)-u_\eps(t,x)\big|\zeta(x)\,dx\Big] \notag \\
   &  \le C(T)\delta^{r}\,\Bigg[ \E\Big[||u_0||_{B_{1, \infty}^{s'}(\R^d)}\Big] + ||\zeta||_{L^\infty(\R^d)} + ||\zeta||_{L^1(\R^d)}\Bigg] 
   + C_2\, \delta^r \E \bigg[||u_\eps(t,\cdot)||_{L^1(\R^d)}\bigg].
 \end{align*}
Let $K_{R}=\{x: |x| \le R\}$. Choose $\zeta \in \mathcal{K}$ such that $\zeta(x)=1$ on $K_R$. Then,
  for $r < \frac{2}{7}$, we have
 \begin{align*}
   \sup_{|y| \le \delta} \E\Bigg[\int_{K_R} \big|u_\eps(t,x+y)-u_\eps(t,x)\big|\,dx\Bigg] \le C(T,R)\,\delta^r,
 \end{align*} 
which completes the proof.
\end{proof}

In view of the well-posedness results from \cite{BaVaWit_2014,BisMajVal}, 
we can finally claim  the existence of entropy solutions for \eqref{eq:levy_stochconservation_laws_spatial} 
that satisfies the fractional BV estimate in Theorem \ref{thm:fractional BV estimate}. In other words, we have the following theorem.

\begin{thm}
Suppose that the assumptions ~\ref{A2}, ~\ref{A3}, ~\ref{A4}, ~\ref{B3}, and ~\ref{B4} hold and the initial data $u_0$ belong to the Besov space $B^{\mu}_{1, \infty} (\R^d)$ for some $\mu \in (\frac27,1)$ and
$ \E \Big[\norm{u_0}^2_{L^2(\R^d)} + \norm{u_0}_{L^1(\R^d)}\Big] < +\infty$.
Then given such initial data $u_0$, there exists an entropy solution of \eqref{eq:levy_stochconservation_laws_spatial} such that for any $t\ge0$,
\begin{align*}
 \E \Big[\norm{u(t,\cdot)}^2_{L^2(\R^d)} \Big] < +\infty.
\end{align*}
Moreover, there exists a constant $C_T^R$ such that, for any $0< t < T$,
\begin{align*}
\sup_{|y| \le \delta} \E\Bigg[\int_{K_R} \big|u(t,x+y)-u(t,x)\big|\,dx\Bigg] \le C_T^R \,\delta^r,
\end{align*}
for some $r \in (0, \frac27)$ and $K_{R}:=\{x: |x| \le R\}$.
\end{thm}
  

\section{Appendix}
\label{appendix}

For the convenience of the reader, we include the proof of the first part of the Theorem~\ref{thm:bv-viscous}, that are 
frequently used in this paper. In what follows, we give a proof of such estimate for a slightly general equation 
\eqref{eq:levy_stochconservation_laws_spatial}, where
noise coefficients depend explicitly on the spatial position $x$.

As we have seen from \cite{BaVaWit_2014,BisMajVal} that under natural assumptions on initial data, flux 
functions, noise coefficients, and the fact that $A: \R \goto \R$ is a nondecreasing Lipschitz continuous function, 
viscous equation \eqref{eq:levy_stochconservation_laws_spatial} has a weak solution $u_\eps$ 
and moreover \eqref{bounds:a-priori-viscous-solution} holds. To that context, under additional assumption on the   
initial data, $u_0 \in L^1(\Omega\times \R^d)$, we show that for fixed $\eps>0$, $u_\eps \in L^1(\Omega\times \Pi_T)$. 
To do this, we proceed as follows: let  us consider a convex, even, approximation of the absolute-value function 
$\beta_\xi$ defined as in Section \ref{sec:tech}: Remark \ref{beta}, \eqref{beta_1} and \eqref{beta_2}.
Then, by applying 
It\^{o}-L\'{e}vy formula to  $\int_{\R^d}\beta_\xi(u_\eps(t,x))\,dx$, we conclude 
\begin{align}
 & \E\Big[ \int_{\R^d} \beta_{\xi}(u_\eps(t,x))\,dx\Big] + \E\Big[ \int_0^t \int_{\R^d} \beta_\xi^{\prime\prime}(u_\eps(s,x)) 
 \big( |\grad G(u_\eps(s,x))|^2 + \eps |\grad u_\eps(s,x)|^2 \big) \,dx \Big] \notag \\
 & \le  \E\Big[ \int_{\R^d} \beta_\xi(u_0(x))\,dx\Big]
 - \E \Big[  \int_0^t \int_{\R^d}  \beta_\xi^{\prime\prime}(u_\eps(s,x))  f(u_\eps(s,x))\cdot \grad u_\eps(s,x) \,dx\,ds\Big] \notag \\
  & \quad  + \E\Big[ \int_0^t \int_{\R^d} \int_{|z|>0}\int_{0}^1 (1-\lambda) \eta^2(x,u_\eps(s,x);z)
  \beta_\xi^{\prime\prime}\big(u_\eps(s,x)+ \lambda \eta(x,u_\eps(s,x);z)\big)\,d\lambda\,m(dz)\,dx\,ds\Big] \notag \\
    &  \hspace{2cm}+ \frac{1}{2} \E\Big[ \int_0^t \int_{\R^d} \sigma^2(x,u_\eps(s,x)) \beta_\xi^{\prime\prime}(u_\eps(s,x))
    \,dx\,ds\Big]. \label{esti:l1-ness-0}
\end{align}
Since $\beta_\xi$  is a convex function, we have from \eqref{esti:l1-ness-0}
\begin{align}
 & \E\Big[ \int_{\R^d} \beta_{\xi}(u_\eps(t,x))\,dx\Big]- \E\Big[ \int_{\R^d} \beta_\xi(u_0(x))\,dx\Big] \notag \\
 & \quad \le - \E \Big[  \int_0^t \int_{\R^d}  \beta_\xi^{\prime\prime}(u_\eps(s,x))  f(u_\eps(s,x))\cdot \grad u_\eps(s,x) \,dx\,ds\Big] \notag \\
  & \quad  + \E\Big[ \int_0^t \int_{\R^d} \int_{|z|>0}\int_{0}^1 (1-\lambda) \eta^2(x,u_\eps(s,x);z)
  \beta_\xi^{\prime\prime}\big(u_\eps(s,x)+ \lambda \eta(x,u_\eps(s,x);z)\big)\,d\lambda\,m(dz)\,dx\,ds\Big] \notag \\
    &  \hspace{2cm}+ \frac{1}{2} \E\Big[ \int_0^t \int_{\R^d} \sigma^2(x,u_\eps(s,x)) \beta_\xi^{\prime\prime}(u_\eps(s,x))
    \,dx\,ds\Big] \notag \\
    & := \mathcal{A}_1(\eps,\xi) + \mathcal{A}_2(\eps,\xi) + \mathcal{A}_3(\eps,\xi).\label{esti:l1-ness-1}
\end{align}
Next, we estimate each of the above terms separately. 
Let us first remark that a simple application of chain-rule implies that $\mathcal{A}_1(\eps,\xi)=0$.
We now move on to the  term $\mathcal{A}_2(\eps,\xi)$. 
In view of assumptions \ref{B3}, and \ref{B4} along with \eqref{eq:approx to abosx}, similar to the estimation 
 $\mathcal{G}$ in Section \ref{sec:apriori+existence} yields
 \begin{align*}
  0 \leq &\eta^2(x,u_\eps(s,x);z)
  \beta_\xi^{\prime\prime}\big(u_\eps(s,x)+ \lambda \eta(x,u_\eps(s,x);z)\big)
  \\
  \leq & (\lambda^*)^2(1 \wedge |z|^2) | u_\eps(s,x)|^2
  \beta_\xi^{\prime\prime}\big(u_\eps(s,x)+ \lambda \eta(x,u_\eps(s,x);z)\big)
  \\
  \leq & \frac{(\lambda^*)^2(1 \wedge |z|^2)}{(1-\lambda\lambda^*)^2} | u_\eps(s,x)+ \lambda \eta(x,u_\eps(s,x);z)|^2
  \beta_\xi^{\prime\prime}\big(u_\eps(s,x)+ \lambda \eta(x,u_\eps(s,x);z)\big)
  \\
  \leq & 4\frac{(\lambda^*)^2(1 \wedge |z|^2)}{(1-\lambda\lambda^*)^2}  \beta_\xi\big(u_\eps(s,x)+ \lambda \eta(x,u_\eps(s,x);z)\big)
  \leq  4\frac{(\lambda^*)^2(1 \wedge |z|^2)}{(1-\lambda\lambda^*)^2}  \beta_\xi\big(|u_\eps(s,x)|+ |\eta(x,u_\eps(s,x);z)|\big)
\\
  \leq & 4\frac{(\lambda^*)^2(1 \wedge |z|^2)}{(1-\lambda\lambda^*)^2}  \beta_\xi\big((1+\lambda^*)|u_\eps(s,x)|\big)
    \leq  4(1 \wedge |z|^2)\frac{(\lambda^*)^2(1+\lambda^*)^2}{(1-\lambda\lambda^*)^2}  \beta_\xi\big(u_\eps(s,x)\big),
\end{align*}
and this implies that 
\begin{align}
\big|\mathcal{A}_2(\eps,\xi)\big|  
  \le&
C\E\Big[ \int_0^t \int_{\R^d} \int_{|z|>0}(1 \wedge |z|^2)\,m(dz)  \beta_\xi\big(u_\eps(s,x)\big)\,dx\,ds\Big]
\leq 
C\E\Big[ \int_0^t \int_{\R^d}  \beta_\xi\big(u_\eps(s,x)\big)\,dx\,ds\Big]
 . \label{esti:a3-l1-ness} 
\end{align}
Again, we use assumption \ref{B31} to conclude 
\begin{align}
  \big|\mathcal{A}_3(\eps,\xi)\big|  \leq 
C\E\Big[ \int_0^t \int_{\R^d}  \beta_\xi\big(u_\eps(s,x)\big)\,dx\,ds\Big]. \label{esti:a4-l1-ness} 
\end{align}
Thus, combining all the above estimates \eqref{esti:a3-l1-ness}-\eqref{esti:a4-l1-ness} in \eqref{esti:l1-ness-1}, we arrive at
\begin{align*}
 \E\Big[ \int_{\R^d} \beta_{\xi}(u_\eps(t,x))\,dx\Big] 
 \leq C\int_0^t\E\Big[  \int_{\R^d}  \beta_\xi\big(u_\eps(s,x)\big)\,dx\Big]\,ds + \E\Big[ \int_{\R^d} \beta_\xi(u_0(x))\,dx\Big], 
\end{align*}
and this implies 
\begin{align}
\E\Big[  \int_{\R^d}  \beta_\xi\big(u_\eps(t,x)\big)\,dx\Big] \leq C
\E\Big[ \int_{\R^d} \beta_\xi(u_0(x))\,dx\Big].
 \label{eq:final-Bv-estimate}
\end{align}
Passing to the limit with respect to $\xi$ yields
\begin{align}
\label{L1-bound}
\E \Big[\int_{\R^d} \big|u_\eps(t,x)\big|\,dx\Big] \le C \E \Big[\int_{\R^d} \big|u_0(x)\big|\,dx \Big].
 \end{align}
 This implies that, $u_\eps \in L^1(\Omega \times \Pi_T)$, for every fixed $\eps>0$.

\medskip

Finally, we finish this section by introducing a special class of functions, which plays a pivotal 
role in our analysis. To that context, let us define the set $\mathcal{K}$ consisting of non-zero 
$\zeta \in C^2(\R^d)\cap L^1(\R^d)\cap L^{\infty}(\R^d)$ for which there is a constant $C$ such that 
$|\grad \zeta(x)| \le C \zeta(x)$, and $|\Delta \zeta(x)| \le C \zeta(x)$. Then we have the following Lemma:

\begin{lem}
\label{imp}
Let $\zeta \in \mathcal{K}$ be any element. Then there exists ${\lbrace \zeta_R \rbrace}_{R>1} \subset C_c^{\infty}(\R^d)$ such that
\begin{align*}
\zeta_R \mapsto \zeta, \grad \zeta_R \mapsto \grad \zeta, \,\,\text{and}\,\,  \Delta \zeta_R \mapsto \Delta \zeta\,\,\text{pointwise in}\,\, \R^d, \,\, \text{as}\,\, R \mapsto \infty
\end{align*}
\end{lem}

\begin{proof}
Note that, modulo a mollification step, we can assume that $\zeta \in C^{\infty}(\R^d)$. Let $\eta \in C_c^{\infty}(\R^d)$ be such that $0\le \eta \le 1$, and $\eta(0)=1$. Let us define $\zeta_R(x)= \zeta(x)\eta(x/R)$. Then a straightforward computation yields
\begin{align*}
\grad \zeta_R(x) &= \grad \zeta(x) \eta(x/R) + \frac1R \zeta(x) \grad \eta(x/R), \\
\Delta \zeta_R(x) &= \Delta \zeta(x) \eta(x/R) + \frac{1}{R^2} \zeta(x) \Delta \eta(x/R)
+  \frac2R \grad \zeta(x) \grad \eta(x/R).
\end{align*}
Taking limit as $R \mapsto \infty$ concludes the proof.
\end{proof}


\begin{thebibliography}{99}
\bibitem{boris_2010}
B.~ Andreianov, M.~ Maliki.
\newblock {  A note on uniqueness of entropy solutions to degenerate parabolic equations in $\R^n$}.
\newblock  {\em NoDeA Nonlinear Differential equations Appl.} 17(1) (2010) 109-118.
	
\bibitem{BaVaWit_2014}
C. Bauzet, G. Vallet and P. Wittbold.
\newblock  A degenerate parabolic-hyperbolic Cauchy problem with a stochastic force.
\newblock{\em J. Hyperbolic Diff. Equ.} 12, (2015), no.3, 501-533.

\bibitem{BaVaWit_JFA}
C. Bauzet, G. Vallet and P. Wittbold.
\newblock The Dirichlet problem for a conservation law with a multiplicative stochastic perturbation.
\newblock{\em J. Funct. Anal.} 266 (2014), no. 4, 2503Ð2545. 

\bibitem{BaVaWit_2012}
C. Bauzet, G. Vallet and P. Wittbold.
\newblock The Cauchy problem for  conservation law with a multiplicative stochastic perturbation.
\newblock{\em J. Hyperbolic Diff. Equ.} 9 (2012), no.4, 661-709.

\bibitem{mostafakarlsen_2004}
M.~ Bendahmane and K.~H.~ Karlsen.
\newblock Renormalized entropy solutions for quasilinear anisotropic degenerate parabolic equations.
\newblock{\em SIAM J. Math. Anal.} 36(2): 405-422, 2004.

\bibitem{mostafakarlsen_2005}
M.~ Bendahmane and K.~H.~ Karlsen.
\newblock Uniqueness of entropy solutions for doubly nonlinear anisotropic degenerate parabolic equations.
\newblock{\em Contemporary Mathematics, Volume 371, Amer. Math. Soc.,} 1-27, 2005.

 \bibitem{BisMaj}
 I.~H.~Biswas and A.~K.~Majee.
 \newblock Stochastic conservation laws: weak-in-time formulation and strong entropy condition. 
 \newblock{\em Journal of Functional Analysis} 267 (2014) 2199-2252.

\bibitem{BisMajKarl_2014}
I.~H.~Biswas, K.~H.~Karlsen and A.~K.~Majee.
\newblock Stochastic balance laws driven by L\'{e}vy noise. 
\newblock{\em J. Hyperbolic Diff. Equ.} 12 (2015), no. 3, 581-654.

\bibitem{BisKoleyMaj}
I. H. Biswas, U.~ Koley, and A.~ K. Majee.
\newblock Continuous dependence estimate for conservation laws with  L\'{e}vy noise.
\newblock{\em J. Diff. Equ.},259 (2015), 4683-4706.


\bibitem{BisMajVal}
I. H. Biswas, A. K. Majee, and G. Vallet.
\newblock On the Cauchy problem of a degenerate parabolic-hyperbolic PDE with L\'{e}vy noise.
\newblock{\em submitted.} ArXiv: http://arxiv.org/pdf/1604.05009.pdf

\bibitem{perthame}
F.~ Bouchnut and B.~ Perthame.
\newblock Kru\v{z}kov's estimates for scalar conservation laws revisited. 
\newblock{\em Trans. Amer. Math. Soc.} 350 (1998), 2847-2870. 



\bibitem{Bustosetalbook}
M.~C. Bustos, F.~Concha, R.~B{\"u}rger, and E.~M. Tory.
\newblock {\em Sedimentation and thickening}, volume~8 of {\em Mathematical
  Modelling: Theory and Applications}.
\newblock Kluwer Academic Publishers, Dordrecht, 1999.
\newblock Phenomenological foundation and mathematical theory.

\bibitem{carrillo_1999}
J.~ Carrillo.
\newblock {Entropy solutions for nonlinear degenerate problems}.
\newblock{\em Arch. Rational Mech. Anal.}, 147(4): 269-361, 1999.

\bibitem{Chen-karlsen_2012}
G.~Q.~ Chen, Q.~ Ding, and K.~ H.~ Karlsen.
\newblock On nonlinear stochastic balance laws.
\newblock{\em  Arch. Rational Mech. Anal.} 204 (2012), no. 3, 707-743.

 \bibitem{chenkarlsen_2005}
 G. Q. Chen, and  K. H. Karlsen.
 \newblock Quasilinear anisotropic degenerate parabolic equations with time-space dependent diffusion coefficients.
 \newblock{ \em Commun. Pure Appl. Anal. 4 (2005), no. 2}, 241-266. 
 

\bibitem{cockburn}
B.~ Cockburn and G.~ Gripenberg.
\newblock Continuous dependence on the nonlinearities of solutions of degenerate parabolic equations. 
\newblock{\em J.  Differential Equations}, 151 (1999), 231-251.

\bibitem{Vovelle2010}
A.~ Debussche and J.~ Vovelle. 
\newblock Scalar conservation laws with stochastic forcing. 
\newblock{\em J. Funct. Analysis}, 259 (2010), 1014-1042. 


\bibitem{martina}
A.~ Debussche and M.~Hofmanov\'{a}, and J.~Vovelle.
\newblock Degenerate parabolic stochastic partial differential equations: quasilinear case
\newblock {\em Annals of Prob.}, to appear.



\bibitem{EspedalKarlsen}
M.~S. Espedal and K.~H. Karlsen.
\newblock Numerical solution of reservoir flow models based on large time step
  operator splitting algorithms.
\newblock In {\em Filtration in porous media and industrial application
  ({C}etraro, 1998)}, volume 1734 of {\em Lecture Notes in Math.}, pages 9--77.
  Springer, Berlin, 2000.
  
  
  
 
\bibitem{karl-resibro-2000}
S.~Evje, K. H. Karlsen and N. H. Risebro.
\newblock A continuous dependence result for nonlinear degenerate parabolic equations with spatially dependent flux function.
\newblock{\em  Hyperbolic problems: theory, numerics, applications, Vol. I, II (Magdeburg, 2000), 337–346, Internat. Ser. Numer. Math.,} 140, 141, Birkhäuser, Basel, 2001.

 \bibitem{fellah}
D~ Fellah and E.~Pardoux.
\newblock Une formule d'It\^{o} dans des espaces de Banach et application.
\newblock {\em Stochastic Analysis and related topics (Silivri, 1990).}, 
Vol. 31 of Progr. Probab., Birkh\"{a}user Boston, Boston, MA, 1992, pp. 197-209.



 

\bibitem{nualart:2008}
J.~ Feng and D.~Nualart.
\newblock Stochastic scalar conservation laws.
\newblock {\em J. Funct. Anal.}, 255(2): 313-373, 2008.



\bibitem{kal-resibro}
K.~H.~ Karlsen and N.~H.~ Risebro.
\newblock On the uniqueness and stability of entropy solutions of nonlinear degenerate parabolic equations with rough coefficients. 
\newblock{\em Discrete Contin. Dyn. Syst.} 9 (2003), 1081-1104. 



\bibitem{KIm2005}
J.~U.~ Kim.
\newblock On a stochastic scalar conservation law, 
\newblock{Indiana Univ. Math. J.}  52 (1) (2003) 227Ã-256. 

\bibitem{Kruzkov}
S.~N. Kru{\v{z}}kov.
\newblock First order quasilinear equations with several independent variables.
\newblock {\em Mat. Sb. (N.S.)}, 81 (123):228--255, 1970.

\bibitem{lions}
P. L.~Lions, B.~Perthame, and E.~Tadmor.
\newblock A kinetic formulation of multidimensional scalar conservation laws and related equations.
\newblock{\em  J. Amer. Math. Soc} 7 (1), 169-191, 1994.



\bibitem{peszat}
S.~Peszat and J.~Zabczyk.
\newblock {\em Stochastic partial differential equations with {L}{\'e}vy
	noise}, volume 113 of {\em encyclopedia of Mathematics and its Applications}.
\newblock Cambridge University Press, Cambridge, 2007.
\newblock An evolution equation approach.



\bibitem{Vallet_2005}
G.~Vallet.
\newblock Dirichlet problem for a degenerated hyperbolic-parabolic equation.
\newblock{\em  Adv. Math. Sci. Appl.} 15 (2005), no. 2, 423--450.
		
\bibitem{Vallet_2008}
G.~Vallet.
\newblock Stochastic perturbation of nonlinear degenerate parabolic problems.
\newblock{\em  Differential Integral equations} 21 (2008), no. 11-12, 1055-1082.

\bibitem{Vallet_2009}
G.~Vallet. and P.~ Wittbold, Petra.
\newblock On a stochastic first-order hyperbolic equation in a bounded
              domain.
\newblock {\em Infin. Dimens. Anal. Quantum Probab. Relat. Top.} 12 (2009), no. 4, 613--651.

\bibitem{Volpert}
A.~I. Vol{\cprime}pert.
\newblock Generalized solutions of degenerate second-order quasilinear
  parabolic and elliptic equations.
\newblock {\em Adv. Differential Equations}, 5(10-12):1493--1518, 2000.

\bibitem{VolpertHudajev}
A.~I. Vol{\cprime}pert and S.~I. Hudjaev.
\newblock The {C}auchy problem for second order quasilinear degenerate
  parabolic equations.
\newblock {\em Mat. Sb. (N.S.)}, 78 (120):374--396, 1969.

\end{thebibliography}
\end{document}